\documentclass{amsart}

\usepackage[english]{babel}
\usepackage{amssymb}
\usepackage{csquotes}
\newcommand{\QQ}{\mathbb{Q}}
\usepackage[all]{xy}
\usepackage{enumitem}

\usepackage{amsmath}

\usepackage{graphicx}

\usepackage[backend=biber,style=numeric,
  doi=false,
  url=false]{biblatex}
\usepackage{bm}

\usepackage{ytableau}
\usepackage{tikz}

\ytableausetup{smalltableaux}
\DeclareMathOperator{\Supp}{Supp}

\usepackage{xcolor}
\usepackage[
  colorlinks=true,
  linkcolor=blue,
  citecolor=red,
  urlcolor=blue
]{hyperref}

\usepackage[a4paper,
  left=30mm,
  right=30mm,
  top=28mm,
  bottom=32mm
]{geometry}

\newtheorem{theorem}{Theorem}[section]
\newtheorem{lemma}[theorem]{Lemma}
\newtheorem{proposition}[theorem]{Proposition}
\newtheorem{corollary}[theorem]{Corollary}
\newtheorem{conjecture}[theorem]{Conjecture}

\theoremstyle{definition}
\newtheorem{definition}[theorem]{Definition}
\newtheorem{example}[theorem]{Example}

\theoremstyle{remark}
\newtheorem{remark}[theorem]{Remark}

\numberwithin{equation}{section}

\newcommand{\NN}{\mathbb{N}}
\newcommand{\ZZ}{\mathbb{Z}}

\newcommand{\bA}{\mathbb{A}}

\newcommand{\rar}{\rightarrow}
\newcommand{\irar}{\hookrightarrow}

\newcommand{\rminor}{(r+1)\times(r+1)}

\newcommand{\exclude}[1]{}

\newcommand{\Sym}{\textnormal{Sym}}

\newcommand{\shape}{\textnormal{sh}}
\newcommand{\rk}{\textnormal{rk}}
\newcommand{\str}{\textnormal{str}}

\newcommand{\len}{\textnormal{len}}

\newcommand{\init}{\textnormal{init}}

\newcommand{\GL}{\textnormal{GL}}

\newcommand{\SD}{\textnormal{SD}}

\title[Symmetric and Isotypic Hilbert Series for Symmetric Ideals]{Symmetric and Isotypic Hilbert Series for Symmetric Ideals}

\author{Henri Breloer}
\address{Department of Mathematics and Statistics, UiT - the Arctic University of Norway, 9037 Tromsø, Norway}
\email{henri.l.breloer@uit.no}

\author{Cordian Riener}
\address{Department of Mathematics and Statistics, UiT - the Arctic University of Norway, 9037 Tromsø, Norway}
\email{cordian.riener@uit.no}

\thanks{This work has been supported by European Union’s HORIZON-MSCA-2023-DN-JD programme under the Horizon Europe (HORIZON) Marie Skłodowska-Curie Actions, grant agreement 101120296 (TENORS)}

\begin{document}
\begin{abstract}
     An ideal in a polynomial ring is symmetric if it is invariant under any permutation of variables. In this paper, we define and study the symmetric and isotypic Hilbert series for symmetric ideals in a polynomial ring with countably many variables. The symmetric Hilbert series is the limit of the Hilbert series of the invariant parts of the finite truncated quotients, while the isotypic Hilbert series records stable multiplicities of irreducible symmetric-group representations for each degree. Our main result proves that, under a mild support condition on the ideal, the symmetric Hilbert series is a rational function. We further show that this rationality extends to the isotypic Hilbert series for every irreducible representation. The proofs of these results rely on the monomial structure of the polynomials within the symmetric ideal, combined with Kostka inversion for the isotypic case.
\end{abstract}
\maketitle

\section{Introduction}\label{sec_intro}
 
The infinite  polynomial ring $K[x_1,x_2,\dots]$ in countably many variables is, both from the viewpoint of commutative algebra and geometry, a remarkably rich object. It is
non-Noetherian and therefore more complicated than a finitely generated polynomial ring, but at the same time it is one of the
simplest infinite-dimensional rings one can write down. The situation becomes even more interesting if one considers symmetry on this ring and finiteness phenomena appear once symmetry is taken
into account. The
first sign of such phenomena is a theorem of Cohen \cite{Coh67,Coh87}, rediscovered by
Aschenbrenner and Hillar \cite{AH07}: every $\Sym(\infty)$-invariant ideal of $K[X]$
satisfies the ascending chain condition. Hillar and Sullivant \cite{HS12} extended this idea of finiteness up to symmetry to define an equivariant Gröbner basis theory with applications in algebraic statistics, and
Draisma \cite{Dra19} placed the phenomenon in a broader frame, proving that any
polynomial functor over an infinite field is topologically Noetherian.

A second strand of the commutative algebra in this infinite setup concerns Hilbert series. For a $\Sym$-invariant chain of
graded ideals $(I_n)_{n\ge 1}$, Nagel and Römer \cite{NR17} introduced the
\emph{equivariant Hilbert series}
\[
H_{(I_n)}(s,t) = \sum_{n\ge 0} H({K[X_n]/I_n};\,t)\,s^n,
\]
and proved it to be rational in $s$ and $t$; an alternative proof via formal languages
is given in \cite{KLS17}. From this rationality one reads off that the dimension of
$K[X_n]/I_n$ grows eventually linearly in $n$ and its degree eventually exponentially,
with the precise rates determined in \cite{NR17,LNNR20}. Further asymptotic invariants
along such chains, for example, regularity, projective dimension, Betti tables, are surveyed in
\cite{JLR20}, and for ideals generated by general forms the asymptotic stability of
Hilbert functions and Betti numbers is established in \cite{SS25}. A parallel theory in
the $\GL_\infty$-equivariant setting of twisted commutative algebras was developed by
Sam and Snowden \cite{SS18}, who proved rationality of Hilbert series for finitely
generated modules. In our work we focus on the $\Sym(\infty)$-structure throughout. The
geometry underlying these results in this setting, the equivariant algebraic geometry, the structure of the equivariant
spectrum has been described by Kummer and Riener \cite{KR25} and  by Nagpal and Snowden \cite{NS21}.

In this paper we look for a more direct analogue of the classical Hilbert series for a
single symmetric ideal $I\subset K[X]$. The ordinary series is unavailable, since each
graded piece of $I$ is either zero or infinite-dimensional. Passing to the truncations
\[
I_1 \subset I_2 \subset \cdots \subset I_n \subset \cdots, \qquad I_n = I\cap K[X_n],
\]
each $I_n$ is $\Sym(n)$-stable with finite-dimensional graded pieces, so we may
decompose
\[
[K[X_n]/I_n]_d = \bigoplus_{\lambda\vdash n} c_\lambda \, V^\lambda
\]
into irreducible $\Sym(n)$-representations, each appearing with multiplicity $c_\lambda$. By the hook length formula
$\dim_K V^\lambda = n!/\prod h_\lambda(i,j)$ (see \cite{Sag01}), these dimensions grow
without bound with $n$, except for the two extremal partitions $\lambda=(n)$ and
$\lambda=(1^n)$, the trivial and sign representations. The equivariant Hilbert series of
\cite{NR17} records this growth. To extract instead the \emph{stable} behaviour as
$n\to\infty$, the natural starting point is the trivial representation, whose
multiplicity does not grow with $n$; more generally, representation stability
\cite{CEF15} suggests tracking the multiplicity of $V^{\mu(n)}$, the representation
indexed by a fixed partition $\mu$ padded in its first part, which again stabilises. The
results of this paper address both settings.

\medskip

\noindent\textbf{The symmetric Hilbert series.}
For each $n$, the subring of $\Sym(n)$-invariants
\[
\bigl(K[X_n]/I_n\bigr)^{\!\Sym}
\;\cong\;
K[X_n]^{\Sym}/I_n^{\Sym}
\]
is a graded module over the ring of symmetric polynomials $K[X_n]^{\Sym}$,
whose degree-$d$ part has finite dimension over $K$ for every $d$.  Its Hilbert
series $H((K[X_n]/I_n)^{\Sym};t)$ is therefore well-defined.  It is not hard to see that 
these series stabilize as $n\to\infty$ coefficient-wise, that is the coefficient of each $t^d$
eventually becomes constant. Indeed, stabilization in each degree follows from the fact that $K[X_n]^{\Sym}$ is
generated by elementary symmetric polynomials $e_1,\dots,e_n$ and stabilizes in
degree $d$ once $n\ge d$, while $I_n^{\Sym}$ grows monotonically and is bounded
above.
Therefore, we can consider a limit series and it is natural to ask   whether this limit is rational. The key
condition we identify for rationality of the limit is:
 
\begin{enumerate}
\item[$(\dagger)$] There exists an $f\in I$ written $f=\sum_{i=1}^N c_i m_i$ with
$c_i\in K^\times$ and $m_i$ monomials, such that
$\Supp(m_i) \not\subset \Supp(m_1)$ for all $i\neq 1$.
\end{enumerate}
 
Here $\Supp(m)$ denotes the set of variables appearing in $m$. 
Condition~$(\dagger)$ holds trivially if $I$ contains a monomial, and its role is to bound the rank of
the leading term in the symmetric limit, forcing the rank decomposition of
$\Lambda/\init(I_\infty)$ into finitely many pieces each handled by Hilbert--Serre.

Moreover, a symmetric ideal generated by polynomials generic in their support must contain a monomial. This is a consequence of \cite{Kre23}. We see that condition $(\dagger)$ holds for almost all symmetric ideals.
 
\begin{theorem}\label{thm_main}
Let $I \subset K[X]$ be a homogeneous symmetric ideal.  Then the limit
\[
\lim_{n\rightarrow\infty} H\big((K[X_n]/I_n)^\Sym;\, t\big)
\]
exists as a formal power series with non-negative integer coefficients.
Moreover, it is a rational function provided condition~$(\dagger)$ holds.
\end{theorem}
 
Condition~$(\dagger)$ also holds under a simple geometric condition on the zero
set of $I$.
 
\begin{corollary}\label{geometric_cor}
Let $I\subset K[X]$ be a homogeneous symmetric ideal, and let $\overline K$ be
an algebraic closure of $K$. If $\overline K$ is uncountable, then the limit
exists and is rational if there exist $p,q\in\NN$ such that
$V_{\overline K}(I) \subset \Delta_p \cup H_q$,
where
\[
\Delta_p = \{(a_1,a_2,\dots)\in \bA^\infty_{\overline{K}} \mid a_i^p = a_j^p
\text{ for all } i,j\}
\]
is the $p$-th root of the diagonal, and
\[
H_q = \bigcup_{\sigma\in \Sym(\infty)}
\sigma\big(\{(a_1,\dots,a_q,0,0,\dots)\mid a_1,\dots,a_q\in\overline K\}\big)
\]
is the union of all $q$-dimensional coordinate subspaces.  The ideals
$I(\Delta_p)$ and $I(H_q)$ are among the $\Sym(\infty)$-invariant prime ideals
characterized in \cite{KR25}.
\end{corollary}
 
Ideals not satisfying~$(\dagger)$ may still yield rational symmetric Hilbert
series; we illustrate the subtleties in Section~\ref{sec_hook_specht}.
 
\medskip
 
\noindent\textbf{Isotypic Hilbert series.}
The methods extend naturally to all irreducible representations in the stable
range.  Fix a partition $\mu\vdash n_0$ with parts
$\mu_1\ge\cdots\ge\mu_\ell>0$ and define its \emph{padding}
$\mu(n) = (n-n_0+\mu_1,\mu_2,\dots,\mu_\ell)\vdash n$ for $n\ge n_0$.  For $n$
large enough the multiplicity $c_{\mu(n),d}(n)$ of $V^{\mu(n)}$ in
$[K[X_n]/I_n]_d$ stabilizes, since in degree~$d$ only padded shapes $\mu(n)$
with $|\mu|\le d$ can appear.  The \emph{isotypic Hilbert series} records these
limits.  The stabilization of individual multiplicities was studied in the
semi-algebraic setting in \cite{BR20}, where polynomial bounds on the
multiplicities of Specht modules in the cohomology of symmetric semi-algebraic
sets are established.
 
\begin{theorem}\label{thm_isotypic_intro}
Let $I\subset K[X]$ be a homogeneous symmetric ideal satisfying condition
$(\dagger)$. Then for every partition $\mu\vdash n_0$ the isotypic Hilbert
series
\[
H^\mu(K[X]/I;\,t)
\;=\;
\sum_{d\ge 0}\Bigl(\lim_{n\to\infty} c_{\mu(n),d}(n)\Bigr)\,t^d
\]
is a rational function.
\end{theorem}
 
Thus condition~$(\dagger)$ guarantees rationality not only of the symmetric
Hilbert series (the trivial representation, $\mu=(n_0)$) but of the stable
multiplicities along every family of irreducible representations.  The
equivariant Hilbert series of Nagel and Römer \cite{NR17} encodes the full
representation structure bivariately; our isotypic series extract individual
stable multiplicities as univariate rational functions, a complementary
perspective.  
\medskip

In this work we work mainly in the case where $\Sym(\infty)$
permutes the variables $x_1,x_2,\dots$. It is possible to consider the more general action when  $\Sym(\infty)$
acts on $k$-tuples, permuting the variables $x_{i,1},\dots,x_{i,k}$ of the infinite
polynomial ring $K[x_{i,j}\mid i\ge 1,\,1\le j\le k]$ by the first index. In this more general setting in the case $k\ge 2$ we have a considerably richer situation. While our methods focus on the $k=1$ case, both the symmetric and isotypic
Hilbert series can be defined in the general setup. We return to the $k$-tuple case in Section~\ref{sec_isotypic}, and highlight with three examples when to hope for rationality in this case. 
 
\noindent\textbf{Structure of the paper.}
In Section~\ref{sec_sym_hilb} we introduce the symmetric Hilbert series and establish its
stabilization. Section~\ref{sec_sym_lim} develops the structure theory of the symmetric
limit $I_\infty$ in the ring of symmetric functions and proves Theorem~\ref{thm_main}.
Section~\ref{sec_hook_specht} treats the hook Specht ideals, a natural family lying
outside condition~$(\dagger)$, and computes their symmetric Hilbert series.
Section~\ref{sec_isotypic} proves Theorem~\ref{thm_isotypic_intro} via block-symmetric
limits and Kostka inversion. In Subsection~\ref{subsec_final_rem} we describe the
generalization to the polynomial ring in $k$-sets of variables and formulate a conjecture relating rationality to the
dimension of the value variety. Finally, we close with open questions in Section~\ref{sec_open_quest}.

\medskip

\section{The Symmetric Hilbert Series}\label{sec_sym_hilb}
 
In this section we fix notation, define the symmetric Hilbert series, and
reformulate the problem in the language of symmetric functions.
Subsections~\ref{subsec_partitions} and~\ref{subsec_order} develop the
combinatorial theory of partitions that will be essential in Section~\ref{sec_sym_lim}.
 
\subsection{The fundamental objects}\label{subsec_fundamental}
 
We work over a field $K$ of characteristic~$0$.  Let $\Sym(n)$ denote the
symmetric group on $n$ elements, and let
\[
\Sym(\infty) := \bigcup_{n\ge 1} \Sym(n)
\]
be the union of all finite symmetric groups, viewed as permutations of $\NN$
that fix all but finitely many elements.  The group $\Sym(n)$ acts on
$K[x_1,\dots,x_n]$ by permuting variables: $\sigma\cdot x_i = x_{\sigma(i)}$.
Passing to the direct limit
\[
K[X] := K[x_i \mid i\ge 1] = \varinjlim K[x_1,\dots,x_n],
\]
we obtain an action of $\Sym(\infty)$ on the infinite polynomial ring.  We
write $K[X_n] = K[x_1,\dots,x_n]$ for the truncation at~$n$.
 
An ideal $I\subset K[X_n]$ is \emph{symmetric} if it is stable under $\Sym(n)$,
and an ideal $I\subset K[X]$ is \emph{symmetric} if it is stable under
$\Sym(\infty)$.  Since every $f\in K[X]$ involves only finitely many variables,
no nonzero polynomial is fixed by all of $\Sym(\infty)$, so there is no
meaningful ring of $\Sym(\infty)$-invariants in $K[X]$.  This is why the limit
construction of Subsection~\ref{subsec_sym_lim} is necessary.
 
For a $K[X_n]$-module $M$ equipped with a $\Sym(n)$-action, we write $M^{\Sym}$
for the submodule of invariants, which is naturally a module over the ring of
symmetric polynomials $K[X_n]^{\Sym}$.  For a symmetric ideal $I\subset
K[X_n]$, the \emph{invariant ideal} $I^{\Sym}$ is an ideal in $K[X_n]^{\Sym}$.
 
The polynomial rings $K[X_n]$ and $K[X]$ are graded by degree, as is
$K[X_n]^{\Sym}$.  Since we work with Hilbert series throughout, we assume all
ideals to be homogeneous; this makes every quotient a graded ring as well.
 
Given a homogeneous ideal $I\subset K[X]$, its \emph{truncation at~$n$} is
$I_n := I\cap K[X_n]$, a homogeneous $K[X_n]$-ideal that is symmetric whenever
$I$ is.

\subsection{The symmetric dimension}
To give a good definition of the symmetric Hilbert series, we first need an intermediate definition.

\begin{definition}
    Let $I\subset K[X]$ be a homogeneous symmetric ideal. Then the \emph{symmetric dimension} of the quotient at degree $d$ is the limit
    \[
    \textnormal{SD}(K[X]/I,d) = \lim_{n\rar\infty}\dim_K \big[ K[X_n]^\Sym / I_n^\Sym \big]_d.
    \]
\end{definition}

\begin{lemma}\label{lemma_fin_sym_dim}
    Let $I\subset K[X]$ be a homogeneous symmetric ideal. Then the symmetric dimension $\SD(K[X]/I,d)$ exists and is finite for all $d\in \NN$.
\end{lemma}
\begin{proof}
    By the fundamental theorem of symmetric polynomials, the ring of symmetric polynomials $K[X_n]^\Sym$ is generated by the first elementary symmetric polynomials $e_1,\dots,e_n$ with $\deg e_k = k$.
    Thus the sequence
    \[
    \dim_K \big[K[X_n]^\Sym\big]_d
    \]
    is constant for $n\geq d$.
    We can bound
    \[
    \dim_K \big[ I_n^\Sym \big]_d \leq \dim_K\big[ K[X_n]^\Sym\big]_d
    \]
    for any $n$,
    and we have inclusions $I_n^\Sym  \irar I_{n+1}^\Sym$ given by $f\mapsto \frac{1}{(n+1)!}\sum_{\sigma\in \Sym(n+1)} \sigma \cdot f$. This shows that $\dim_K \big[I_n^\Sym\big]_d$ is eventually constant just as the dimension of $\big[K[X_n]^\Sym\big]_d$.
    Thus the limit of the difference of the two dimensions, $\SD(K[X]/I,d)$, exists and is finite.
\end{proof}

\begin{definition}\label{def_sym_hilb_series}
Let $I\subset K[X]$ be a homogeneous symmetric ideal.  
The \emph{symmetric Hilbert series} of $K[X]/I$ is the formal power series
\[
H^{\Sym}(K[X]/I;\,t)
    := \sum_{d\geq 0} \SD(K[X]/I,d) \cdot t^d.
\]
\end{definition}

With this definition, the limit of Hilbert series of truncations in Theorem \ref{thm_main} is exactly the symmetric Hilbert series,
\[
\lim_{n\rar\infty} H(K[X_n]^\Sym/I_n^\Sym;t) = H^\Sym(K[X]/I; t).
\]
Now Lemma \ref{lemma_fin_sym_dim} already solves the existence part of the theorem.

\subsection{The symmetric limit}\label{subsec_sym_lim}
The key to the existence of the symmetric dimension was the embedding of invariant ideals. We can take this idea to its conclusion and construct a graded object whose degree-$d$ part has dimension exactly equal to the symmetric dimension. This will be the \emph{symmetric limit} of $I$. We construct it as a direct limit.

\begin{definition}\label{def_part_sym}
    For each $n$, we have a \emph{partial symmetrization} map $\phi_n$ on $K[X]$ given by
    \[
    \phi_n(f) = \frac{1}{n!}\sum_{\sigma\in \Sym(n)} \sigma\cdot f.
    \]
\end{definition}

\begin{lemma}\label{lemma_part_sym}
    For every $n,m$ we have
    \begin{enumerate}
        \item $\phi_n$ is linear and respects grading.
        \item $\phi_n \circ \phi_m = \phi_m \circ \phi_n = \phi_n$ if $n \geq m$.
        \item $\phi_n(sf) = s\phi_n(f)$ if $s$ is invariant under $\phi_n$.
    \end{enumerate}
\end{lemma}
\begin{proof}
    $(1)$ is clear, since the action of $\Sym(n)$ on $K[X]$ is linear and does not change degree.
    For $(2)$, we can see that $\phi_n\circ\phi_m = \phi_m\circ\phi_n$ by reordering the sums. The image of $\phi_n$ is $\Sym(n)$-invariant and $\Sym(m)\subset\Sym(n)$, thus $\phi_m\circ\phi_n = \phi_n$.
    For $(3)$ we compute
    \[
    \phi_n(sf) = \frac{1}{n!}\sum_{\sigma\in\Sym(n)} \sigma \cdot (sf) = s\Big(\frac{1}{n!}\sum_{\sigma\in\Sym(n)} \sigma\cdot f\Big) = s\phi_n(f). 
    \]
\end{proof}

Points $(1)$ and $(2)$ show that partial symmetrization, restricted to $K[X_n]^\Sym$ and to $I_n^\Sym$, respectively, define a directed system.
This allows us to define their direct limit with respect to partial symmetrization. Note that since we only have maps of graded vector spaces, the resulting limit objects are only graded vector spaces. We will address this problem in the next section.

\begin{definition}\label{def_sym_lim}
    We call the direct limit of the ring of symmetric polynomials \emph{the ring of symmetric functions},
    \[
    \Lambda := \lim_{\longrightarrow} K[X_n]^\Sym.
    \]
    Given a homogeneous symmetric ideal $I\subset K[X]$, we call the direct limit of the invariant part of its truncations \emph{the symmetric limit},
    \[
    I_\infty := \lim_{\longrightarrow} I_n^\Sym.
    \]
    We call the corresponding canonical homomorphism \emph{the total symmetrization},
    \[
    \phi_\infty: K[X] \rar \Lambda.
    \]
\end{definition}

From the definition, it is clear that $\dim_K [\Lambda/I_{\infty}]_d = \SD(K[X]/I,d)$, thus the symmetric Hilbert series of $K[X]/I$ is exactly $H(\Lambda/I_\infty;t)$.

\subsection{Partitions and monomial symmetric functions}\label{subsec_partitions}

We recall the notation for integer partitions that will be used throughout.

\begin{definition}\label{def_rank}
A partition $\lambda\vdash d$ is a weakly decreasing sequence
$\lambda=(\lambda_1,\dots,\lambda_\ell)$ of nonnegative integers with sum $d$,
identified up to trailing zeros.  
We write $\len(\lambda)$ for the number of nonzero parts.  
The dual partition $\lambda^\perp$ is defined by
\[
(\lambda^\perp)_i := \#\{j \mid \lambda_j\ge i\}.
\]
The \emph{rank} of $\lambda$ is
$\rk(\lambda)=\max\{i \mid \lambda_i\ge i\}$; by convention, the zero partition
has rank~$0$.
\end{definition}

For a monomial $m=x_1^{\alpha_1}x_2^{\alpha_2}\dots x_n^{\alpha_n}$, its
\emph{shape} $\shape(m)$ is the partition obtained by sorting
$\alpha = (\alpha_1,\alpha_2,\dots,\alpha_n)$ in descending order.

\begin{definition}\label{def_mon_sym_poly}
Let $\lambda\vdash d$ with $\len(\lambda) \leq n$.  
The (normalized) monomial symmetric polynomial $m_\lambda\in K[X_n]^{\Sym}$ is the normalized sum of all
monomials of shape $\lambda$:
\[
m_\lambda := \frac{1}{N}\sum_{\shape(m)=\lambda} m,
\]
where $N$ is the number of monomials of shape $\lambda$ in $K[X_n]$. 
\end{definition}

Such a monomial symmetric polynomial evaluates to $1$ when all coordinates are $1$,
\[
m_\lambda(\underbrace{1,\dots,1}_{n\textnormal{ times}}) = 1
\]
and the image of $m_{\lambda}(x_1,\dots,x_n)$ in $K[X_{n+k}]^\Sym$ under $\phi_{n+k}$ is exactly $m_\lambda(x_1,\dots,x_{n+k})$.

\begin{example}
    For partitions $\lambda = (2,1)$, $\mu = (4)$, we have monomial symmetric polynomials
    \[
    m_\lambda(x_1,x_2) = \frac{1}{2}(x_1^2x_2 + x_1x_2^2), \quad m_\lambda(x_1,x_2,x_3) = \frac{1}{6}(x_1^2x_2 + x_1x_2^2 + x_1^2x_3 + x_1x_3^2 + x_2^2x_3 + x_2x_3^2)
    \]
    and
    \[
    m_\mu(x_1,x_2) = \frac{1}{2}(x_1^4 + x_2^4), \quad m_\mu(x_1,x_2,x_3) = \frac{1}{3}(x_1^4 + x_2^4 + x_3^4).
    \]
\end{example}

For $n\ge d$, the polynomials $\{m_\lambda \mid \lambda\vdash d, \len(\lambda)\le
n\}$ form a basis of $[K[X_n]^{\Sym}]_d$.  
Passing to the direct limit yields the \emph{monomial symmetric functions} (MSFs).  
Identifying each MSF with its corresponding partition, we obtain a vector space basis for
\[
\Lambda = \langle \lambda\mid \lambda\text{ a partition}\rangle_K,
\]
the graded vector space of symmetric functions.

In general, $\phi_N(m_\lambda m_\mu) \neq \phi_N(m_\lambda)\phi_N(m_\mu)$ for monomial symmetric polynomials in $K[X_n]^\Sym$ and $N> n$. Writing the product as $n$ increases, we see that the coefficients of monomial symmetric polynomials corresponding to partitions with greatest length will eventually dominate. However, we can introduce two artificial products on $\Lambda$ with desired properties. We explore this further in the next section.

\begin{definition}\label{def_joint_disjoint_prod}
For partitions $\lambda$ and $\mu$:
\begin{enumerate}
\item The \emph{disjoint product} $\lambda\vee\mu$ is the decreasing
rearrangement of all parts of $\lambda$ and $\mu$.
\item The \emph{joint product} is the entry-wise sum
\[
\lambda\wedge\mu = (\lambda_1+\mu_1,\;\lambda_2+\mu_2,\dots),
\]
and satisfies $\lambda\wedge\mu = (\lambda^\perp\vee \mu^\perp)^\perp$.
\end{enumerate}
\end{definition}

Joint and disjoint product are analogously defined on MSFs and extend linearly to two product structures on $\Lambda$, which respect the grading.

\begin{example}
\begin{enumerate}
\item $(3,1)^\perp=(2,1,1)$;
\item $(3,2)\vee(4,1,1)=(4,3,2,1,1)$;
\item $(4,1)\wedge(3,3,2)=(7,4,2)$.
\end{enumerate}
\end{example}

\begin{definition}\label{def_dominance}
    For two partitions $\mu,\nu\vdash n$, we say $\mu$ \emph{dominates} $\nu$ and write $\mu \trianglerighteq \nu$ if
    \[
    \mu_1+\dots+\mu_k \geq \nu_1+\dots+\nu_k\ \textnormal{for all}\ k.
    \]
    This defines a partial order on partitions.     
\end{definition}
The dominance order will be relevant for Section \ref{sec_isotypic}.

\subsection{A combinatorial order on partitions}\label{subsec_order}

To describe structural properties of $I_\infty$, we require a combinatorial
encoding of partitions.

\begin{definition}\label{def_string}
Write a partition as $\lambda=(\lambda_1^{e_1},\dots,\lambda_k^{e_k})$ with
distinct nonzero parts $\lambda_i$ and multiplicities $e_i>0$.  
The \emph{string} of $\lambda$ is the word
\[
\str(\lambda)
    := 1^{e_1}0^{\lambda_1-\lambda_2}
       1^{e_2}0^{\lambda_2-\lambda_3}
       \cdots
       1^{e_k}0^{\lambda_k}
    \in \{0,1\}^*.
\]
\end{definition}

\begin{example}
The partitions $\lambda=(3,3,1)$ and $\mu=(5,2,1,1)$ have
\[
\str(\lambda)=110010,
\qquad
\str(\mu)=100010110.
\]
The string traces the boundary of the Young diagram: $1$ corresponds to a down
step and $0$ to a left step.
\end{example}

Let $\mathcal P = \{0,1\}^*/\!\sim$, where $s\sim 0s$ and $s\sim s1$. It is the set of all finite sequences made from letters $1$ and $0$ identified up to leading $0$'s and trailing $1$'s.

\begin{lemma}
The map $\str$ induces a bijection between partitions and $\mathcal P$.
\end{lemma}

\begin{proof}
Every class in $\mathcal P$ has a unique representative beginning with $1$ and
ending with $0$.  
Such strings correspond bijectively to sequences
$1^{e_1}0^{d_1}\dots1^{e_k}0^{d_k}$, and its unique preimage is given by
\[
(\lambda_1^{e_1},\dots,\lambda_k^{e_k}),
\qquad
\lambda_k=d_k,\;\lambda_i=d_i+\lambda_{i+1}.
\]
\end{proof}

\begin{definition}\label{def_partition_divisibility}
For partitions $\lambda$,$\mu$, the partition $\lambda$ \emph{divides} $\mu$, written $\lambda\mid\mu$, if $\str(\lambda)$ is a
subsequence of $\str(\mu)$.
\end{definition}

This defines a partial order on partitions.  By Higman’s lemma \cite{Hig52}
it is a well partial order: in every infinite sequence of partitions, some
earlier partition divides a later one. This fact will be crucial for the structural
analysis in Section~\ref{sec_sym_lim}.

\section{Structure of the Symmetric Limit}\label{sec_sym_lim}

The symmetric limit
\[
I_\infty = \varinjlim I_n^{\Sym} \subset \Lambda
\]
associated to a homogeneous symmetric ideal $I\subset K[X]$ inherits many ``local'' product structures from the finite truncations $I_n^\Sym$ in $K[X_n]^\Sym$. In this section we study its global structure with respect to the joint and disjoint product, which will lead to the proof of the main theorem. These two product structures on $\Lambda$ manage to contain all necessary information when we restrict to initial MSFs after choosing a suitable partition order.

\subsection{The joint-disjoint closure} We first characterize subspaces of the ring of symmetric functions which are closed under both the joint and disjoint product.

\begin{definition}\label{def_joint_disjoint_closure}
Let $S\subset \Lambda$ be any subset.  
We define the \emph{disjoint closure} $\Lambda \vee S$ as the $(\Lambda,\vee)$-ideal
generated by $S$. Furthermore, we define the \emph{joint closure} $\Lambda\wedge S$ as the $(\Lambda,\wedge)$-ideal
generated by $S$. We say that a subset $S$ is \emph{joint–disjoint closed} if it is both a
$(\Lambda,\vee)$-ideal and a $(\Lambda,\wedge)$-ideal.
\end{definition}

\begin{lemma}\label{lemma_joint_disjoint_closed}
Let $\mathcal{B}\subset \Lambda$ be a set of MSFs. Then
\[
\Lambda\vee(\Lambda\wedge \mathcal{B})
    = \Lambda\wedge(\Lambda\vee\mathcal{B}).
\]
In particular, this set is joint–disjoint closed.
\end{lemma}

\begin{proof}
Let $f\in\Lambda$.  By definition, $f\in \Lambda\vee(\Lambda\wedge\mathcal B)$
if and only if it can be written as a finite sum
\[
f = \sum_{i=1}^N c_i\bigl(\nu_i\vee(\mu_i\wedge\lambda_i)\bigr),
\]
for some $N\in\NN$, $c_i\in K$, partitions $\nu_i,\mu_i$, and
$\lambda_i\in\mathcal B$.

Similarly, $f\in \Lambda\wedge(\Lambda\vee\mathcal B)$ if and only if it is a
finite sum of expressions of the form
\[
f = \sum_{i=1}^{N'} c_i'
\bigl(\nu_i' \wedge (\mu_i' \vee \lambda_i')\bigr),
\]
with $\lambda_i'\in\mathcal B$ and partitions $\nu_i',\mu_i'$.

Thus it suffices to show that for any partitions $\mu,\nu$ and any partition
$\lambda$ there exist partitions $\mu',\nu'$ such that
\[
\mu\vee(\nu\wedge\lambda) = \nu'\wedge(\mu'\vee\lambda).
\]

Fix $\lambda$ and consider its string $\str(\lambda)$.  
The disjoint product $\lambda\vee\mu$ corresponds to inserting $1$'s into a
representative of $\str(\lambda)$ in $\mathcal P$, and every way of inserting
$1$'s into such a representative is realized by some disjoint product
$\lambda\vee\mu$.  
Similarly, the joint product $\lambda\wedge\mu$ corresponds to inserting $0$'s
into a representative of $\str(\lambda)$, and every insertion of $0$'s arises in
this way.  
This follows from the description of the joint product in terms of dual
partitions and the fact that the string of $\lambda^\perp$ is obtained from
$\str(\lambda)$ by reversing order and swapping $0$ and $1$.

Now fix a string $a$.  If a string $b$ can be obtained from $a$ by first
inserting $1$'s and then inserting $0$'s, then $b$ can be obtained equally well
by first inserting $0$'s and then inserting $1$'s.  
Translating back to partitions, this shows that for any $\mu,\nu$ there exist
$\mu',\nu'$ with
\[
\mu\vee(\nu\wedge\lambda) = \nu'\wedge(\mu'\vee\lambda),
\]
as required.
\end{proof}

\subsection{A total monomial order in the ring of symmetric functions}\label{subsec_revlex_order}
We extend the partial order on partitions defined by divisibility to a total order $\geq$ on the corresponding MSFs in $\Lambda$. This is an analogue of a monomial ordering in a polynomial ring. In a similar manner to a Gröbner basis helping to find the Hilbert series of a typical ideal, we will find a finite number of elements of $I_\infty$ that, in a certain sense, serve the same role as a Gröbner basis. As it turns out, the most beneficial order is a reverse-length lexicographical order.

Let $\lambda,\mu$ be two integer partitions. We say that $\lambda \geq \mu$ if
\begin{enumerate}[label=(\roman*)]
    \item $\len(\lambda)  <  \len(\mu)$ or
    \item $\len(\lambda) = \len(\mu)$ and the first nonzero entry of $(\lambda_i-\mu_i)_i$ is positive.
\end{enumerate}

Note that this implies both
\[
\lambda \geq \mu \Rightarrow \lambda \vee \nu \geq \mu\vee\nu 
\]
and
\[
 \lambda \geq \mu \Rightarrow \lambda \wedge \nu \geq \mu\wedge\nu\quad \textnormal{if}\  \len(\nu) \leq \len(\lambda).
\]
for any partition $\nu$. Thus the disjoint product preserves the monomial ordering. Any joint product $\lambda\wedge\nu$ with $\ell_1 = \len(\nu)> \len(\lambda) = \ell_2$ can be replaced by a joint product with the first $\ell_2$ entries of $\nu$ and a disjoint product of the rest. Explicitly
\[
\lambda\wedge \nu = (\lambda \wedge \nu')\vee \nu''
\]
where $\nu' = (\nu_1,\dots,\nu_{\ell_2})$ and $\nu'' = (\nu_{\ell_2+1},\dots,\nu_{\ell_1})$. This new operation preserves the monomial ordering as well. Hence $\geq$ is indeed an extension of divisibility.

\begin{definition}\label{def_initial_MSF}
    Given any symmetric function $f\in\Lambda$, there is a unique way to write it as a finite sum of MSFs. The \emph{initial} MSF of $f$, denoted $\init(f)$, is the $\geq$-maximal such MSF.

    Given a graded subspace $V\subset\Lambda$, its \emph{initial space} is the graded vector space
    \[
    \init(V) := \langle \init(f)\mid f\in V \rangle_K.
    \]
\end{definition}

\subsection{The structure of the initial space}
Taking the initial space of a symmetric limit $\init(I_\infty)$ makes it into a joint-disjoint closed set. This leads to a finite generation result up to joint-disjoint closure.

\begin{proposition}\label{prop_initial_ideal_closed_cond}
    Let $I \subset K[X]$ be a homogeneous symmetric ideal and $I_\infty$ its symmetric limit in $\Lambda$. Then $\init(I_\infty)$ is joint-disjoint closed.
\end{proposition}

\begin{proof}
    In both cases it suffices to check multiplication of MSFs, since $\init(I_\infty)$ is, by definition, generated by MSFs. Let $\lambda\in \init(I_\infty)$ be any partition in the initial space. Then there exists an $f\in I_\infty$ such that $\init(f) = \lambda$. We write it as 
    \[
    f = \lambda + \sum_{i=1}^N c_i \lambda_i
    \]
    where $c_i \in K$ and $\lambda_i$ partitions for all $i=1,\dots,N$.
    Fix any partition $\mu = (\mu_1,\dots,\mu_k)$. Since $I_\infty$ is a direct limit, there exists $n$ large enough such that we can find a preimage of $f$,
    \[
    f_{(n)} = m_\lambda + \sum_{i=1}^N c_im_{\lambda_i} \in I_n^\Sym
    \]
    where $m_\lambda$ and $m_{\lambda_i}$ are monomial symmetric polynomials in $K[X_n]^\Sym$. In particular, $f_{(n)} \in I_n$ and therefore
    \[
    g_{(n)} := \phi_{n+k}(f_{(n)}\cdot x_{n+1}^{\mu_1}\dots x_{n+k}^{\mu_k}) \in I_{n+k}.
    \]
    But $g_{(n)}$ is also a symmetric polynomial, thus $g_{(n)}\in I_{n+k}^\Sym$. It is given by
    \[
    g_{(n)} = m_{\mu \vee\lambda} + \sum_{i=1}^N c_i m_{\mu\vee \lambda_i}
    \]
    since the support of the monomial $x_{n+1}^{\mu_1}\dots x_{n+k}^{\mu_k}$ is disjoint from the support of all monomials of $f_{(n)}$. Because $I_\infty$ is a direct limit of the $I_n^\Sym$, we can conclude that
    \[
    g := \phi_\infty(g_{(n)}) = \mu\vee \lambda + \sum_{i=1}^N c_i\cdot\mu\vee\lambda_i \in I_\infty.
    \]
    The disjoint product is order preserving, thus $\init(g) = \mu\vee\lambda$ and therefore $\mu\vee \lambda \in \init(I_\infty)$ as required.

    For the second part, take $f$ with $\init(f) = \lambda$ as above and fix a partition $\mu$ of length $k$. We want to show that $\lambda \wedge \mu \in \init(I_\infty)$. If $k = \len(\mu) > \len(\lambda) = \ell$, we write the joint product as
    \[
    \lambda \wedge \mu = (\lambda\wedge \underbrace{(\mu_1, \dots, \mu_\ell)}_{=: \mu'})\vee \underbrace{(\mu_{\ell+1},\dots,\mu_k)}_{=: \mu''}
    \]
    as discussed in the previous section.
    By the previous part of the proof, if $\lambda \wedge \mu' \in \init(I_\infty)$, then also $(\lambda \wedge \mu')\vee \mu''$. Thus we may assume that $\len(\mu) \leq \len(\lambda)$.
    
    For $n$ large enough, $f_{(n)}\in I_n^\Sym$ and $m_\mu\in K[X_n]^\Sym$ as above. Then $f_{(n)}\cdot m_\mu \in I_n^\Sym$. Recall that the monomial symmetric polynomials are normalized sums of monomials of their corresponding shape.
    Writing the product as a sum of monomial symmetric polynomials
    \[
    f_{(n)}\cdot m_\mu = m_\lambda m_\mu + \sum_{i=1}^N c_i m_{\lambda_i}m_\mu
    \]
    we see that the monomial symmetric polynomial corresponding to the maximal partition is exactly $c\cdot m_{\lambda\wedge\mu}$ for some fraction $c$, since $\len(\mu) \leq \len(\lambda)$.
    Hence $\init(f\cdot \mu) = c\cdot \lambda\wedge\mu \in \init(I_\infty)$ as required.
\end{proof}

\begin{lemma} \label{initial_ideal_const_dim}
    Let $V \subset \Lambda$ be a graded linear subspace. Then $\dim_K{[V]}_d = \dim_K{[\init(V)]}_d$ for all $d\in \NN$.
\end{lemma}
This is a direct consequence of \cite[Theorem 15.3]{Eis95}.

\begin{proposition}\label{prop_special_fin_gen}
    Let $V \subset \Lambda$ be a graded linear subspace that is joint-disjoint closed and $\init(V) = V$. Then there exists a finite set $\mathcal{B} \subset V$ such that
    \[
    \Lambda\vee(\Lambda\wedge\mathcal{B}) = \Lambda\wedge(\Lambda\vee\mathcal{B}) = V.
    \]
\end{proposition}

\begin{proof}
    Let $\mathcal B =  \{ \lambda \in V\mid \lambda\textnormal{ a MSF}\}$ be the set of all MSFs in $V$. By assumption $V = \Lambda\vee(\Lambda\wedge\mathcal{B})$.

    Let $(\lambda^{(i)})_{i\in\NN}$ be any enumeration of the countably many elements of $\mathcal{B}$. We remove $\lambda^{(j)}$ from the list, if there exists $i< j$ with $\lambda^{(i)}\mid\lambda^{(j)}$ using the notion of divisibility from Definition \ref{def_partition_divisibility}. This gives rise to a new sequence $(\lambda^{(i)})_{i\in I}$ and generating set $\mathcal{B}' := \{ \lambda^{(i)}\mid i\in I\}$. By transitivity of divisibility, these are independent of the order of removal of the $\lambda^{(j)}$. According to our definition of divisibility, we still have
    \[
    \Lambda\vee(\Lambda\wedge\mathcal{B}') = V.
    \]
    Divisibility is a well partial order and $(\lambda^{(i)})_{i\in I}$ is a sequence which no longer has any divisible pairs $\lambda^{(i)} \mid \lambda^{(j)}$ for $i < j$, thus it must be a finite sequence. This makes $\mathcal{B}'$ a finite joint-disjoint generating set for $V$.
\end{proof}

This general result on finite generation combined with the fact that $\init(I_\infty)$ is indeed joint-disjoint closed from Proposition~\ref{prop_initial_ideal_closed_cond} gives us a finite generation result for all symmetric limits of ideals in $K[X]$ independent of condition~$(\dagger)$.

The finite generation result is too general to show rationality of the symmetric Hilbert series. There exist many examples of finitely generated joint-disjoint closed spaces $V\subset \Lambda$ that do not have a rational symmetric Hilbert series. We demonstrate this with a simple counterexample in the next section. We have to use an approach that does not rely on finite generation for the proof of Theorem \ref{thm_main}.

\subsection{Partition rank decomposition}
The proof of Theorem \ref{thm_main} relies on the notion of rank of a partition from Definition \ref{def_rank} and the joint-disjoint closedness of the initial space of the symmetric limit. We first give a technical characterization of condition $(\dagger)$.

\begin{lemma}\label{lemma_dagger_block_eq}
    Let $I\subset K[X]$ be a homogeneous symmetric ideal. $I$ satisfies condition $(\dagger)$ if and only if, for some $a,b > 0$, $\init(I_\infty)$ contains a MSF of block shape
    \[
    (b^a) = (\underbrace{b,\dots,b}_{a\textnormal{ times}}).
    \]
\end{lemma}

\begin{proof}
    Suppose $I$ satisfies condition $(\dagger)$.
    Recall that the monomial order $\geq$ is the reverse-length lexicographical order. MSFs in $\Lambda$ of shorter length will always be greater than longer MSFs using $\geq$. 
    The ideal $I$ satisfies condition~$(\dagger)$, so we find an $f\in I$ with $f = m_1 + \sum_{i=2}^N c_im_i$ with $\Supp (m_i) \not\subset \Supp (m_1)$ for $i\geq 2$ as prescribed. By symmetry of $I$, we may assume that $\Supp (m_1) = \{1,\dots,a\}$ and all other monomials of $f$ have support containing some number greater than $a$. We write $m_1 = x_1^{\alpha_1}\dots x_a^{\alpha_a}$ and set $b = 1+ \max\{ \alpha_i\mid i=1,\dots,a\}$. Then
    \[
    f' := \phi_\infty( f \cdot x_1^{b-\alpha_1}\dots x_a^{b-\alpha_a}) \in I_\infty
    \]
    has initial MSF $(b^a)$, since all other MSFs appearing in $f'$ are, by our construction, of greater length.

    Now suppose $\init(I_\infty)$ contains a MSF of shape $(b^a)$. Then there exists a homogeneous $f\in I_\infty$ with initial term $(b^a)$. Let $\lambda$ be the shape of some term appearing in $f$. Since $(b^a) \geq \lambda$ in the reverse-length lexicographical order, we have $\len(\lambda) \geq a$. If $\len(\lambda) = a$, we can conclude that $\lambda = (b^a)$ since this is the smallest shape of degree $ab$ and length $a$.

    Finally, let $f_{(n)} \in I_n$ be a lift of $f$ such that the total symmetrization is $\phi_\infty(f_{(n)}) = f$. It exists for some $n$ large enough by the properties of the direct limit. The polynomial $f_{(n)}$ must have a term $m_1$ of shape $\shape(m_1) = (b^a)$. Suppose there exists any other term $m$ of $f_{(n)}$ such that $\Supp(m)\subset \Supp(m_1)$, then $\len(\shape(m)) \leq a$. Thus by the previous argument, $\len(\shape(m)) = a$ and therefore $\shape(m) = \shape(m_1) = (b^a)$.
    
    There is only one monomial of this shape for the fixed support $\Supp{(m_1)}$, thus $m = m_1$ and condition $(\dagger)$ is satisfied.
\end{proof}

\begin{proof}[Proof of Theorem \ref{thm_main}]
    Let $I\subset K[X]$ be a homogeneous symmetric ideal for which condition~$(\dagger)$ holds. We denote the symmetric limit of $I$ as usual by $I_\infty$.

    By Lemma \ref{initial_ideal_const_dim}, we have
    \[
    H^\Sym(K[X]/I;t) = H(\Lambda/I_\infty; t) = H(\Lambda/\init(I_\infty);t).
    \]

    We grade $\Lambda/\init(I_\infty)$ by partition rank.
    Consider the vector space
    \[
    R_m := \langle \lambda \notin \init(I_\infty) \mid \rk (\lambda) = m \rangle_K
    \]
    and note that we have a vector space decomposition
    \[
    \Lambda/\init(I_\infty) = \bigoplus_{m = 0}^\infty R_m.
    \]
    
    With a slight shift in perspective, each $R_m$ can be seen as a finitely generated module over the polynomial ring
    \[
    A_m := K[y_1,z_1,\dots,y_m,z_m]
    \]
    in $2m$ variables.
    Given any MSF $\lambda \in R_m$, we define $A_m$-multiplication by linearly extending the multiplication defined for monomials in $A_m$ and MSFs in $\Lambda$,
    \[
    y_1^{c_1} y_2^{c_2}\dots y_m^{c_m} z_1^{d_1} z_2^{d_2}\dots z_m^{d_m} \cdot \lambda := {(m^{d_m},\dots,2^{d_2},1^{d_1})}\vee \Big({(m^{c_m},\dots,2^{c_2},1^{c_1})^\perp}\wedge \lambda\Big).
    \]
    
    By Proposition~\ref{prop_initial_ideal_closed_cond}, $\init(I_\infty)$ is closed under this $A_m$-multiplication. We therefore only need to check that $A_m$-multiplication does not change the rank of a rank $m$ partition and of course that this module structure is well-defined.
    
    Multiplication with a monomial in the $y_i$ only changes the first $m$ entries of $\lambda$ and $\rk(\lambda) = m$ guarantees that multiplication with a monomial in the $z_i$ can never change the first $m$ entries, as it can only add new entries smaller than $m+1$. Thus neither multiplication with a monomial in the $y_i$ nor multiplication with a monomial in the $z_i$ can increase the rank of a MSF $\lambda$. Neither joint nor disjoint multiplication decrease the rank of a partition. This shows that $R_m$ is closed under the defined multiplication.
    
    Finally, we have $f\cdot(g\cdot \lambda) = (fg) \cdot \lambda$ for any $f,g \in A_m$. It suffices to check this on monomials. The joint and disjoint products are commutative, so we only have to check $y_i\cdot(z_j\cdot\lambda)=z_j\cdot(y_i\cdot\lambda)$. This equality holds, since $y_i$ only changes the first $m$ entries of $\lambda$ and multiplication with $z_j$ disjointly adds entries smaller than $m+1$. But the rank of $\lambda$ is $m$. Thus $z_j$-multiplication only affects entries of $\lambda$ after the first $m$ entries.
    All other properties making $R_m$ an $A_m$-module are clear. It is generated by the partition $(m^m)$ which is the rank $m$ partition of minimal degree.
 
    By our assumption, $I$ satisfies condition $(\dagger)$. Thus Lemma \ref{lemma_dagger_block_eq} shows that $\init(I_\infty)$ contains a partition $(b^a)$ of block shape. Extending by joint and disjoint product, we find that the MSF $(M^M)$ is an element of $\init(I_\infty)$ for all $M\geq \max\{a,b\}$.
    In particular $R_M = \{0\}$ for all such $M$.
    Thus we can write
    \[
    \Lambda/\init(I_\infty) = \bigoplus_{m=0}^{\max\{a,b\}-1} R_m,
    \]
     a finite sum of finitely generated modules over polynomial rings.
     We have
     \[
     H(\Lambda/\init(I_\infty);t) = \sum_{m=0}^{\max\{a,b\}-1} H(R_m;t).
     \]
     By the Hilbert--Serre theorem, the Hilbert series of a graded, finitely generated module over a polynomial ring, such as $R_m$, is a rational function. A finite sum of rational functions is rational.
\end{proof}

The proof of Corollary \ref{geometric_cor} is now a simple computation, showing that the geometric condition implies $(\dagger)$. 
\begin{proof}[Proof of Corollary \ref{geometric_cor}]
    Let $I\subset K[X]$ be a homogeneous symmetric ideal with vanishing set
    \[
    V_{\overline{K}}(I) \subset \Delta_p \cup H_q
    \]
    where $\overline{K}$ is the algebraic closure of $K$. The vanishing ideal of $\Delta_p$ is
    \[
    I(\Delta_p) = (x_i^p-x_j^p\mid i,j \geq 1)
    \]
    and the vanishing ideal of $H_q$ is given by
    \[
    I(H_q) = (x_{i_0}\dots x_{i_q}\mid 1 \leq i_0 < i_1 < \dots < i_q ),
    \]
    the ideal generated by all square-free monomials of sufficient length.
    Let
    \[
    f := x_1\dots x_q x_{q+1}^p - x_1\dots x_q x_{q+2}^p \in I(\Delta_p \cup H_q)
    \]
    Since the cardinality of $\overline{K}$ is uncountable by assumption, it is greater than the cardinality of $\NN$ which is the index set for our variables in $K[X]$. We can therefore use Hilbert's Nullstellensatz in infinite-dimensional space, see \cite{Lan52}. This shows, just like in the finite dimensional case, that $I(\Delta_p \cup H_q) \subset \sqrt{I}$, so $f \in \sqrt{I}$.

    Consequently $f^r\in I$ for some $r > 0$, and $m = x_1^r\dots x_q^rx_{q+1}^{pr}$ is the only term of $f^r$ with support contained in $\{1,\dots,q+1\}$; all other monomials are divisible by $x_{q+2}$. Hence $I$ satisfies $(\dagger)$ and we can apply Theorem \ref{thm_main} showing that the symmetric Hilbert series of $K[X]/I$ is rational.
\end{proof}

\subsection{Examples and a counterexample} We demonstrate that it is possible to compute the symmetric Hilbert series by hand in simple cases. We also give an example of a finitely generated joint-disjoint closed initial space in $\Lambda$ which has a non-rational Hilbert series. This does not constitute a counterexample to rationality of the symmetric Hilbert series as this space is not the initial space of a symmetric limit.

\begin{example}
Let $I^{(1)} = I(\Delta_1)$ be the vanishing ideal of the diagonal. Thus
\[
I^{(1)} = (x_i-x_j\mid i \neq j).
\]
Clearly $K[X_n]/I_n^{(1)} \cong K[x]$ and so $H\big((K[X_n]/I_n^{(1)})^\Sym;t\big) = \frac{1}{1-t}$ for all $n\geq 1$. So for the limit, we find
\[
H^\Sym(K[X]/I^{(1)};t) = \frac{1}{1-t}
\]
as well.

Let $I^{(2)} = I(H_N)$ for some $N\in \NN$. Then $I^{(2)}\subset K[X]$ is generated by all squarefree monomials $x_{i_0}\dots x_{i_N}$ for $1 \leq i_0 < i_1 <\dots< i_N$. Thus its symmetric limit $I_\infty \subset \Lambda$ is generated by all partitions with $\len(\lambda) \geq N+1$ and we already have $\init(I_\infty) = I_\infty$. A simple counting argument shows
\[
H^\Sym(K[X]/I^{(2)};t) = H(\Lambda/I_\infty; t)  = \prod_{k=1}^{N}\frac{1}{1-t^k}.
\]
\end{example}

\begin{example}
    Fix $\lambda = (2,1)$ and let $V = \Lambda\wedge(\Lambda\vee\{\lambda\})$. This space is generated by a single element, up to joint-disjoint closure, yet the quotient $\Lambda/V$ does not have a rational Hilbert series. Joint-disjoint finite generation is not enough to guarantee a rational Hilbert series.
    
    We have $\str(\lambda) = 1010$. Using the subsequence condition of divisibility, we see that $\lambda \nmid \mu$ if either $\mu$ is the empty partition, or $\str(\mu) = 1^a0^b$ for some $a,b \geq 1$. Such a partition is of form $\mu = (b^a)$ and degree $ab$. Thus $\Lambda/V = \langle\mu\mid \str(\mu) = 1^a0^b, a,b\geq 0 \rangle_K$ and the Hilbert series of the quotient is
    \[
    H(\Lambda/V; t) = 1 + \sum_{a,b = 1}^\infty t^{{ab}} = 1 + \sum_{a=1}^\infty \frac{t^a}{1-t^a}
    \]
    The last part is the well-known Lambert series, which allows the transformation
    \[
    H(\Lambda/V; t) = 1 + \sum_{d=1}^\infty \sigma_0(d)t^d
    \]
    where $\sigma_0(d)$ denotes the number of divisors of $d$.
    This is known not to be a rational function, since the number-of-divisors function does not follow a linear recurrence relation.
\end{example}

\section{Hook Specht Ideals}\label{sec_hook_specht}
We next highlight a natural class of ideals that do not satisfy condition $(\dagger)$, but have a rational symmetric Hilbert series nonetheless. These are the Specht ideals corresponding to hook partitions. This family of ideals is of particular interest, since \cite[Theorem 2]{MRV21} shows that any symmetric ideal $I\subset K[X]$ contains the hook Specht ideals corresponding to $(\infty,1^r)$ for all sufficiently large $r$.

\subsection{The main theorem for hook Specht ideals}
The \emph{Vandermonde matrix} for values $a_0,\dots,a_n$ is
\[
V(a_0,\dots,a_n) = \begin{pmatrix}
1 & a_0 & a_0^2 & \cdots & a_0^n \\
1 & a_1 & a_1^2 & \cdots & a_1^n\\
1 & a_2 & a_2^2 & \cdots & a_2^n \\
\vdots & \vdots & \vdots & \ddots & \vdots \\
1 & a_n & a_n^2 & \cdots & a_n^n \\
\end{pmatrix}.
\]
The \emph{Vandermonde polynomial} is the determinant
\[
\det V(x_1,\dots,x_n) = \prod_{1\leq i<j\leq n} (x_j-x_i).
\]
 
This lets us define
\[
I = \big(\det V(x_{i_0},\dots,x_{i_r})\mid i_0,\dots,i_r \geq 1\big) \subset K[X],
\]

the \emph{hook Specht ideal} corresponding to the infinity partition $(\infty,1^r)$. This name comes from the fact that the truncated ideal $I_n = I\cap K[X_n]$ is the Specht ideal corresponding to the hook partition $\lambda = (n-r,1^r)$ in $K[X_n]$. The vanishing set of $I$ is the set of all sequences in $\bA^\infty_K$ that attain at most $r$ different values in $K$. These ideals do not satisfy condition $(\dagger)$ for $r\geq 2$. 
We will nonetheless prove rationality of their symmetric Hilbert series:

\begin{theorem}\label{thm_hook_specht_main}
    Let $I$ be the hook Specht ideal corresponding to the partition $(\infty,1^r)$. Then the symmetric Hilbert series of $K[X]/I$ is the rational function
    \[
    H^\Sym(K[X]/I;t) = \frac{1}{1-t}\sum_{k=0}^{r-1}\prod_{i=1}^k\frac{t^{2i}}{(1-t^i)(1-t^{i+1}) }.
    \]
\end{theorem}

As a consequence, we obtain a polynomial bound on the growth of the coefficients of all symmetric Hilbert series.

\begin{corollary}\label{cor_sym_polynomial_growth}
    Let $I\subset K[X]$ be any nonzero homogeneous symmetric ideal. Then the corresponding symmetric dimensions have at most polynomial growth. Precisely
    \[
    \SD(K[X]/I,d) = \mathcal{O}(d^{2r-1})
    \]
    where $r$ is the minimal degree of polynomials appearing in $I$.
\end{corollary}

\begin{proof}
    Let $I\subset K[X]$ be any nonzero homogeneous symmetric ideal. Let $f\in I$ be a nonzero homogeneous polynomial of minimal degree $r := \deg(f)$. Denote by $\ell$ the smallest number such that $f\in K[X_\ell]$.
    
    Let $J\subset K[X]$ be Specht ideal corresponding to the hook partition $(\infty,1^{r})$. By \cite[Theorem 2]{MRV21}, $J_n \subset I_n$ for all $n \geq r+\ell$. In particular
    \[
    [J_n^\Sym]_d \subset [I_n^\Sym]_d\quad\textnormal{ if } n\geq r+\ell
    \]
    for all degrees $d$.
    Thus comparing dimensions of the quotients of $K[X_n]^\Sym$ with the above ideals for $n\rar \infty$,
    \[
    \SD(K[X]/I, d) \leq \SD(K[X]/J,d).
    \]
    Finally, by the explicit formula of Theorem \ref{thm_hook_specht_main}, $\SD(K[X]/J,d) = \mathcal{O}(d^{2r-1})$.
\end{proof}

\subsection{The space of Hankel matrices}
Recall that a Hankel matrix is a matrix with constant counter diagonals.
Given a Vandermonde matrix with variables as entries, we consider its normalized product with its transpose
\[
\frac{1}{n}V(x_1,\dots,x_n)^TV(x_1, \dots,x_n) = \begin{pmatrix}
    1 & p_1 & \cdots & p_{n-1}\\
    p_1 & p_2 & \cdots & p_{n}\\
    \vdots & \vdots & \ddots & \vdots \\
    p_{n-1} & p_{n} & \cdots & p_{2n-1}
    \end{pmatrix}.
\]
The resulting matrix is a Hankel matrix in the power sums $p_k = \frac{1}{n}(x_1^k+\dots+x_n^k)$. Generalizing this idea, we let $V_n$ be the matrix with $n$ rows and infinitely many columns such that $\big(V_n\big)_{i,j} = x_i^{j-1}$. It is an infinite, non-square, Vandermonde matrix. Then
\[
A_n = \frac{1}{n}V_n^TV_n
\]
is the infinite Hankel matrix with entries $\big(A_n\big)_{i,j} = p_{i+j-2}$ in $K[X_n]^\Sym$. We denote its limit for $n\rar \infty$ with entries in $\Lambda$ by $A_\infty$. Note that $p_k$ is the monomial symmetric polynomial corresponding to the length one partition $(k)$, thus its limit is exactly the MSF $(k)$.

We fix $I\subset K[X]$ to be the hook Specht ideal corresponding to the partition $(\infty,1^r)$ and denote $I_n^\Sym = I\cap K[X_n]^\Sym$ and $I_\infty$ as its symmetric limit. We want to show that the $\rminor$ minors of $A_\infty$ generate the symmetric limit $I_\infty$ as a $(\Lambda,\vee)$-ideal. This will let us compute its Hilbert series.

It is a theorem of Kronecker that an infinite Hankel matrix has finite rank if its entries are given by a univariate rational function. A precise statement where the rank of the Hankel matrix is given depending on the degree of the rational function can be found in section 4 of \cite{Pel98}. This lets us see that all $\rminor$ minors of $A_n$ vanish on a point in $\bA^n_K$ if its coordinates attain at most $r$ distinct values. This coincides with the vanishing set of $I_n^\Sym$. Leading this idea to its conclusion requires working over an algebraically closed field, however. We will use more algebraic methods for a general proof.

\subsection{The symmetric limit of Hook Specht ideals}
Recall that $\phi_n$ denotes the partial symmetrization from Definition \ref{def_part_sym} and $\phi_\infty$ denotes the total symmetrization.

\begin{lemma}\label{lemma_fin_trunc_generation}
    Let $I \subset K[X]$ be any homogeneous symmetric ideal and let $k$ be any integer such that the $\Sym(\infty)$-orbits $\Sym(\infty)\cdot I_k$ generate $I$ as a $K[X]$-ideal. Then $\phi_\infty(I_{k}^\Sym)$ generates $I_\infty$ as a $(\Lambda,\vee)$-ideal.
\end{lemma}
\begin{proof}
    By assumption, polynomials of the form $\sigma (fm)$ for $f\in I_k$, $\sigma \in \Sym(\infty)$ and $m = x_{1}^{a_{1}}\dots x_{n}^{a_{n}}$ linearly generate $I$. We may assume that $a_1,\dots,a_k = 0$ without loss of generality by absorbing variables $x_1,\dots,x_k$ into $f\in I_k$. Partial symmetrization of truncations gives $\phi_n(I_n) = I_n^\Sym$. Thus by $I_n \subset I$ for all $n\geq 1$ and the properties of the direct limit, we get $\phi_\infty(I) = I_\infty$. Since total symmetrization $\phi_\infty$ is linear, we can complete the proof by checking that $\phi_\infty(\sigma(fm))$ lies in the $(\Lambda,\vee)$-ideal generated by $\phi_\infty(I_{k}^\Sym)$.

    Since total symmetrization is $\Sym(\infty)$-invariant, $\phi_\infty(\sigma(fm)) = \phi_\infty(fm)$. The support of $m$ is disjoint from the first $k$ variables, thus $m$ is invariant under $\Sym(k)$-action. We can write $\phi_k(fm) = \phi_k(f)m$ by Lemma \ref{lemma_part_sym} and thus $\phi_\infty(fm) = \phi_\infty(\phi_k(f)m)$. Let $\lambda$ be the shape of $m$. Again, since the support of $m$ is disjoint from $\{1,\dots,k\}$, we have
    \[
    \phi_\infty(\phi_k(f)m) = \phi_\infty(\underbrace{\phi_k(f)}_{\in I^\Sym_k})\vee \lambda.
    \]
    This lies in the $(\Lambda,\vee)$-ideal generated by $\phi_\infty(I_k^\Sym)$ as required. 
\end{proof}

\begin{proposition}\label{prop_hook_base_case_gen}
    Let $I$ be the hook Specht ideal corresponding to $(\infty, 1^r)$. Then $I_{r+1}^\Sym$ is linearly generated by the $\rminor$ minors of the Hankel matrix $A_{r+1}$.
\end{proposition}
\begin{proof}
    We write $a(i_1,\dots,i_{r+1})$ for the $\rminor$ minor of $V_{r+1}$ of columns $i_1<\dots<i_{r+1}$. It is an alternating polynomial.
    We further write $v = a(1,2,\dots,r+1)$ for the Vandermonde polynomial in the first $r+1$ variables. Basic theory of alternating polynomials
    shows that the polynomials $a(i_1,\dots,i_{r+1})$ linearly generate the space of alternating polynomials in $r+1$ variables and are all divisible by $v$. The quotients are the Schur polynomials as given by Jacobi's bialternant formula.

    Thus, combining these results, we see that the polynomials $a(i_1,\dots,i_{r+1})\cdot a(j_1,\dots,j_{r+1})$ linearly generate the $K$-space $I^\Sym_{r+1}$. This product of two minors is the same as the minor of rows $i_1,\dots,i_{r+1}$ and columns $j_1,\dots,j_{r+1}$ of the product $V_{r+1}^T V_{r+1} = (r+1)\cdot A_{r+1}$.

    Hence the minors of the Hankel matrix $A_{r+1}$ indeed linearly generate $I_{r+1}^\Sym$.
\end{proof}
    
\begin{corollary}\label{cor_hook_gen}
    Let $I$ be the hook Specht ideal corresponding to $(\infty, 1^r)$. Then its symmetric limit $I_\infty$ is generated as a $(\Lambda,\vee)$-ideal by the $\rminor$ minors of the infinite Hankel matrix $A_\infty$. 
\end{corollary}
\begin{proof}
    By definition of the hook Specht ideal, $\Sym(\infty)\cdot I_{r+1}$ generates all of $I$ as a $K[X]$-ideal. Thus we can apply Lemma \ref{lemma_fin_trunc_generation} and find that $\phi_\infty (I_{r+1}^\Sym)$ generates $I_\infty$ as a $(\Lambda,\vee)$-ideal. Now Proposition \ref{prop_hook_base_case_gen} gives us exactly that $I_{r+1}^\Sym$ is linearly generated by the $\rminor$ minors of $A_{r+1}$.

    Let $m$ be such a minor of rows $i_1,\dots,i_{r+1}$ and columns $j_1,\dots,j_{r+1}$. Then it is divisible by $x_1^{i_1+j_1-2}$. Since $m$ is a symmetric polynomial, it is therefore divisible by the monomial $x_1\dots x_{r+1}$ if either $i_1 > 1$ or $j_1 > 1$. In particular, all terms that are not divisible by $x_1\dots x_{r+1}$ must cancel out when computing the determinant
    \[
    m = \begin{vmatrix}
    p_{i_1+j_1-2} & \cdots & p_{i_1 + j_{r+1}-2}\\
    \vdots & \ddots & \vdots \\
    p_{i_{r+1}+j_1-2} & \cdots & p_{i_{r+1}+j_{r+1}-2}
    \end{vmatrix}
    \]
    and replacing $p_k$ by $\frac{1}{r+1}(x_1^k+\dots+x_{r+1}^k)$ in the resulting expression.
    We apply this to the Leibniz formula
    \[
    m = \sum_{\sigma\in\Sym(r+1)} \textnormal{sgn}(\sigma) \prod _{k=1}^{r+1} p_{i_k+j_{\sigma(k)}-2}
    \]
    to simplify the products, as we only need to consider terms divisible by $x_1\dots x_{r+1}$.
    Thus we can write $m$ in terms of $x_1,\dots,x_{r+1}$ as
    \[
    m = \frac{(r+1)!}{(r+1)^{r+1}}\sum_{\sigma \in \Sym(r+1)} \textnormal{sgn}(\sigma) \phi_{r+1}\Big(\prod_{k=1}^{r+1} x_k^{i_k+j_{\sigma(k)}-2}\Big).
    \]
    Note that this expression still holds for $i_1 = j_1 = 1$.
    The limit of this expression in $\Lambda$ is the minor of $A_\infty$ under disjoint product up to constant factor $\frac{(r+1)!}{(r+1)^{r+1}}$. Thus $I_\infty$ is indeed generated by the $\rminor$ minors of $A_\infty$.
\end{proof}

\subsection{Computing the Hilbert series}

We managed to describe the symmetric limit of hook Specht ideals as a determinantal ideal of an infinite Hankel matrix. This is highly useful as determinantal ideals of Hankel matrices are well understood. In particular, a Gröbner basis for such an ideal was constructed in \cite{Con98}. This Gröbner basis generalizes to the case of infinitely many variables without difficulty and lets us describe the Hilbert series of $\Lambda/I_\infty$ by counting monomials in $\Lambda/\init(I_\infty)$. This is an infinite sum of rational functions, but with substantial effort, it is possible to simplify this to the expression seen in Theorem \ref{thm_hook_specht_main}.

We adopt the notation from \cite{Con98}.

\begin{definition}
    Let $<_1$ be the strict partial order on $\NN$ defined by
    \[
    i <_1 j \quad\textnormal{if and only if} \quad i + 1 < j
    \]
    and $>_L$ be the reverse-length lexicographical monomial order on $(\Lambda,\vee)$ seen as a polynomial ring in $p_1,p_2,p_3,\dots$ induced by 
    $p_1>p_2>p_3>\dots$.
    
    A sequence of numbers $s = (s_1,s_2,\dots,s_k)$ is called a \emph{$<_1$-chain} of length $k$ if $s_1 >_1 s_2 >_1 \dots >_1 s_k$.
    We say that a partition or MSF $\lambda$ is \emph{divisible} by a $<_1$-chain of length $k$ if it contains a subsequence of length $k$ that is a $<_1$-chain. Note that this includes the trailing zeroes of every partition we usually suppress in our notation.
\end{definition}

\begin{proposition}\label{prop_groebner_basis}
    The $\rminor$ minors of $A_\infty$ form a Gröbner basis for $I_\infty$ with respect to $>_L$.
\end{proposition}
\begin{proof}
    Let $R = K[y_0,\dots,y_n]$ be some polynomial ring in $n+1$ variables and $J$ be the ideal generated by all $\rminor$ Hankel minors in $y_0,\dots,y_n$. We fix a degree lexicographical monomial order on $R$ induced by $y_0 > y_1 > \dots > y_n$. Following \cite[Proposition 3.4]{Con98} and the subsequent discussion, we find that the $\rminor$ Hankel minors form a Gröbner basis for $J$ and in particular $\init(J)$ is generated by squarefree monomials $y_{i_1}\dots y_{i_{r+1}}$ where the indices form a $<_1$-chain, $i_1 <_1 \dots <_1 i_{r+1}$.

    We consider the homomorphism $\psi :R \longrightarrow \Lambda$ given by $\psi(y_i) := p_i$. It is not surjective as $\psi(y_0) = 1$, but it fixes the monomial order for fixed degree: Let $m,m'$ be two monomials in $[R]_d$ with $m < m'$. Then $\psi(m) <_L \psi(m')$. This homomorphism maps Hankel minors to Hankel minors.

    Let $m = p_1^{a_1}\dots p_\ell^{a_\ell}$ be any monomial in $\init(I_\infty)$. For $n$ large enough, we find a preimage of $m$ in $R$. It is of the form $m' = y_0^Ny_1^{a_1}\dots y_\ell^{a_\ell}$ and for $N$ large enough, it must lie in $\init(J)$. Thus $m'$ is divisible by a monomial of the form $y_{i_1}\dots y_{i_{r+1}}$ where the indices form a $<_1$-chain. Hence we can conclude that the MSF $\psi(m')$ is divisible by a $<_1$-chain of length $r+1$. All initial monomials of Hankel minors in $\Lambda$ are divisible by $<_1$-chains. Thus the initial monomials of the Hankel minors in $\Lambda$ already generate $\init(I_\infty)$. They must form a Gröbner basis.
\end{proof}

As we have seen in the proof, the initial ideal $\init(I_\infty)$ with respect to $>_L$ is linearly generated by all partitions which are divisible by a $<_1$-chain of length $r+1$.
We can stratify $\Lambda/\init(I_\infty)$ as follows: let $P_k$ be the $K$-space of partitions in $\Lambda$ that are divisible by a length $k$ $<_1$-chain, but not a length $k+1$ $<_1$-chain. Then if $I$ is the hook Specht ideal corresponding to $(\infty, 1^r)$,
\[
    \Lambda/\init(I_\infty) = \bigoplus_{k=1}^r P_k.
\]

A partition in $P_k$ is of the form
\[
\lambda = ((s_1 +1)^{a_1},s_1^{b_1},(s_2+1)^{a_2},s_2^{b_2},\dots, (s_{k-1}+1)^{a_{k-1}},s_{k-1}^{b_{k-1}},1^{a_{k}})
\]
for a $<_1$-chain $s_1 >_1 \dots >_1 s_{k-1} >_1 0$. Here the exponents $a_i$ can be any nonnegative integers and the $b_i$ any positive integers. This description is bijective, which allows us to compute the Hilbert series

\[
H(P_k;t) = \sum_{\bm{s},a,b} t^{a_k}\prod_{i=1}^{k-1} t^{(s_i+1)a_i + s_ib_i},
\]
where we sum over all $<_1$-chains $\bm s$, $s_1 >_1\dots >_1 s_{k-1} >_1 0$ and $a_i \geq 0$, $b_i\geq 1$.
Summing over the $a_i$ and $b_i$ is independent of the rest of the summation, so we can simplify
\[
H(P_k;t) = \frac{1}{1-t}\sum_{\bm{s}}\prod_{i=1}^{k-1} \frac{t^{s_i}}{(1-t^{s_i})(1-t^{s_i+1})}
\]
 
To keep the equations concise, we write
\[
B(s,k) = \prod_{i=0}^{k-1} (1-t^{s+i}),\qquad A(s,k) = \frac{t^{ks + k(k-1)}}{B(s,k)}.
\]

\begin{proposition}\label{prop_induction_formula}
    Let
    \[
    H(s,k) := \sum_{s_1 >_1\dots>_1s_k>_1 s}  \prod_{i=1}^k\frac{1}{1-t} \big(A(s_i,1) - A(s_i+1,1)\big)
    \]
    This infinite sum of rational functions can be simplified to the rational function
    \[
    H(s,k) = \frac{A(s+2,k)}{B(1,k)}.
    \]
\end{proposition}

We will need the following identity several times.
\begin{equation}\label{equ_comb_identity}
    \frac{1-t^{k}}{B(s,k+1)} = \frac{1}{B(s,k)} - \frac{t^{k}}{B(s+1,k)}
\end{equation}
It can be shown by a simple computation
\[
\frac{1-t^{k}}{B(s,k+1)} = \frac{(1-t^{s+k})-t^{k}(1-t^s)}{B(s,k+1)} = \frac{(1-t^{s+k})}{B(s,k+1)} - \frac{t^{k}(1-t^s)}{B(s,k+1)} = \frac{1}{B(s,k)} - \frac{t^{k}}{B(s+1,k)}.
\]

Combining these results with the previous discussion allows us to finally prove Theorem \ref{thm_hook_specht_main}. 
\begin{proof}[Proof of Theorem \ref{thm_hook_specht_main}]
    Let $I$ be the hook Specht ideal corresponding to $(\infty,1^r)$. Then Corollary \ref{cor_hook_gen} shows that it is generated by the $\rminor$ minors of the infinite Hankel matrix $A_\infty$ as an $(\Lambda,\vee)$-ideal and Proposition \ref{prop_groebner_basis} confirms that this generating set is already a Gröbner basis with respect to the monomial order $>_L$. With the subsequent discussion, we see that we can therefore decompose $\Lambda/\init(I_\infty) = \bigoplus_{k = 1}^r P_k$.
    Using identity \eqref{equ_comb_identity}, we see that $H(P_k;t) = \frac{1}{1-t}H(0,k-1)$. Thus using Proposition \ref{prop_induction_formula}, we find
    \begin{align*}
        H^\Sym(K[X]/I;t) &= \sum_{k=1}^r H(P_k;t) \\
        &= \frac{1}{1-t}\sum_{k=0}^{r-1}H(0,k) \\
        &= \frac{1}{1-t}\sum_{k=0}^{r-1} \frac{t^{(k+1)k}}{B(1,k)B(2,k)} \\
        &= \frac{1}{1-t}\sum_{k=0}^{r-1}\prod_{i=1}^k\frac{t^{2i}}{(1-t^i)(1-t^{i+1}) }.
    \end{align*}
\end{proof}

\begin{proof}[Proof of Proposition \ref{prop_induction_formula}]
    We use induction over $k$. The case $k=0$ is clear since $H(s,0) = 1$ for all $s$. Now assume
    \[
    H(s,k') = \frac{A(s+2,k')}{B(1,k')}
    \]
    for all $k'<k$.
    Factoring out common terms of the form $\frac{1}{1-t}\big(A(s_i,1) - A(s_i+1,1)\big)$ in the infinite sum expression of $H(s,k)$ lets us write it as a nested sum
    \[
    H(s,k) = \sum_{s_k >_1 s}\frac{A(s_k,1) - A(s_k+1,1)}{1-t}\sum_{s_{k-1}>_1s_{k}}\frac{A(s_{k-1},1) - A(s_{k-1}+1,1)}{1-t}\dots\sum_{s_1 >_1s_2} \frac{A(s_1,1) - A(s_1+1,1)}{1-t}.
    \]
    This can be simplified to
    \[
    H(s,k) =  \sum_{s_k >_1 s}\frac{A(s_k,1) - A(s_k+1,1)}{1-t}H(s_k,k-1) =  \sum_{s_k >_1 s}\frac{A(s_k,1) - A(s_k+1,1)}{1-t} \cdot\frac{A(s_k+2,k-1)}{B(1,k-1)},
    \]
    where the second equality follows from the induction assumption. We now use identity \eqref{equ_comb_identity} and the definition of $A(s,k)$ and $B(s,k)$ to simplify

    \begin{align}
        H(s,k) &= \sum_{s_k >_1 s}\frac{A(s_k,1) - A(s_k+1,1)}{1-t} \cdot\frac{A(s_k+2,k-1)}{B(1,k-1)} \label{eq:step1}\\ 
        &= \sum_{s_k >_1 s}\frac{t^{s_k}}{(1-t^{s_k})(1-t^{s_k+1})}\cdot \frac{t^{(s_k+2)(k-1) +(k-1)(k-2)}}{B(1,k-1)B(s_k+2,k-1)} \label{eq:step2}\\
        &= \sum_{s_k >_1 s}\frac{t^{ks_k + k(k-1)}}{B(1,k-1)B(s_k,k+1)} \label{eq:step3}\\ 
        &= \frac{1}{B(1,k)}\sum_{s_k >_1 s}\frac{(1-t^k)t^{ks_k + k(k-1)}}{B(s_k,k+1)}  \label{eq:step4}\\
        &= \frac{1}{B(1,k)}\sum_{s_k >_1 s}\Big( \frac{t^{ks_k + k(k-1)}}{B(s_k,k)} - \frac{t^{k(s_k+1) + k(k-1)}}{B(s_k+1,k)} \Big) \label{eq:step5}\\ 
        &= \frac{1}{B(1,k)}\sum_{s_k >_1 s}\big(A(s_k,k) - A(s_k+1,k)\big) \label{eq:step6} \\ 
        &= \frac{A(s+2,k)}{B(1,k)} \label{eq:step7}
    \end{align}
    We go through the simplification step by step. At \eqref{eq:step1} we use the induction assumption. At \eqref{eq:step2} we use identity \eqref{equ_comb_identity} on the left fraction and the definition of $A(s,k)$ on the right fraction. At \eqref{eq:step3} we combine the fractions, simplify the exponent of the numerator and absorb $(1-t^{s_k})(1-t^{s_k+1})$ into $B(s_k+2,k-1)$ to get $B(s_k,k+1)$. At \eqref{eq:step4} we expand the fraction by $\frac{1-t^k}{1-t^k}$, then absorb the denominator into $B(1,k-1)$ to get $B(1,k)$. This is independent of $s_k$, thus we pull it out of the sum.
    At \eqref{eq:step5} we use identity \eqref{equ_comb_identity} which allows us to rewrite terms inside of the sum in terms of $A(s,k)$ in \eqref{eq:step6}. Finally, at \eqref{eq:step7} we use a telescoping series to obtain the final result and complete the induction.
\end{proof}

\section{Isotypic Hilbert Series}\label{sec_isotypic}

The symmetric Hilbert series captures the contribution of the trivial representation to the
graded pieces of the quotient $K[X_n]/I_n$.  In this section we extend Theorem~\ref{thm_main}
to arbitrary irreducible representations in the \emph{stable range}.  The key idea is to replace
full symmetrization by partial symmetrization with respect to Young subgroups, and then to
recover individual multiplicities via Kostka inversion.

\subsection{Padded partitions and stable multiplicities}

Fix a partition $\mu\vdash n_0$ with $\ell$ parts.  For every $n\ge n_0$ we define the
\emph{padded partition}
\[
   \mu(n) \;=\; (n-n_0+\mu_1,\,\mu_2,\,\dots,\,\mu_\ell) \;\vdash\; n.
\]
Note that $\mu(n_0)=\mu$ and that $\mu(n)$ is obtained from $\mu$ by adding $n-n_0$ to the
first part.  For $\mu=(n_0)$ the trivial partition, $\mu(n)=(n)$ corresponds to the trivial
representation, and we recover the setting of the preceding sections.

\begin{definition}\label{def_isotypic_hs}
Let $I\subset K[X]$ be a homogeneous symmetric ideal.  The \emph{isotypic Hilbert series}
of $K[X]/I$ associated to $\mu$ is the formal power series
\[
   H^{\mu}(K[X]/I;\,t)
   \;=\; \sum_{d\ge 0}\Bigl(\lim_{n\to\infty}c_{\mu(n),d}(n)\Bigr)\,t^d,
\]
where $c_{\mu(n),d}(n)$ denotes the multiplicity of the irreducible representation $V^{\mu(n)}$ in the
$\Sym(n)$-representation $[K[X_n]/I_n]_d$.
\end{definition}

The existence of the limit and the rationality of the series will follow from the next two
subsections.

\subsection{Block-symmetric limits}

The \emph{Young subgroup} associated to $\mu(n)$ is
\[
   S_{\mu(n)} \;=\;
   \Sym(\{1,\dots,n\!-\!n_0\!+\!\mu_1\})
   \;\times\;
   \Sym(\{n\!-\!n_0\!+\!\mu_1\!+\!1,\dots,n\!-\!n_0\!+\!\mu_1\!+\!\mu_2\})
   \;\times\;\cdots\;\times\;
   \Sym(\{n\!-\!\mu_\ell\!+\!1,\dots,n\}).
\]
As $n$ increases, only the first factor grows; the remaining $\ell-1$ factors are isomorphic
to fixed symmetric groups $\Sym(\mu_2),\dots,\Sym(\mu_\ell)$.

The ring of $S_{\mu(n)}$-invariants decomposes as
\[
   K[X_n]^{S_{\mu(n)}}
   \;\cong\;
   K[X_{n-n_0+\mu_1}]^{\Sym(n-n_0+\mu_1)}
   \;\otimes_K\; R_\mu,
\]
where
\[
   R_\mu
   \;=\;
   K[X_{\mu_2}]^{\Sym(\mu_2)}\otimes_K\cdots\otimes_K K[X_{\mu_\ell}]^{\Sym(\mu_\ell)}
\]
is a polynomial ring in $\mu_2+\cdots+\mu_\ell = n_0-\mu_1$ generators of degrees
$1,\dots,\mu_2;\; 1,\dots,\mu_3;\;\dots;\; 1,\dots,\mu_\ell$.  Here we use the fundamental
theorem of symmetric polynomials in each block.

\begin{definition}\label{def_block_sym_limit}
Let $I\subset K[X]$ be a homogeneous symmetric ideal.  For each $n\ge n_0$, denote by
$I_n^{S_{\mu(n)}}$ the subspace of $S_{\mu(n)}$-invariants in $I_n$.

Define the \emph{shifted inclusion} $\iota_s\colon K[X_n]\rar K[X_{n+1}]$ by
$\iota_s(x_k) = x_{k+1}$.  For $f\in I_n$, the image $\iota_s(f) = f(x_2,\dots,x_{n+1})$
lies in $I_{n+1}$: the substitution $x_k\mapsto x_{k+1}$ is realised by the permutation
$\tau = (1,\,2,\,\dots,\,n,\,n+1)\in\Sym(\infty)$, and $I$ is $\Sym(\infty)$-invariant.

Define the \emph{block-symmetric partial symmetrization}
$\phi_n^\mu\colon K[X_n]\rar K[X_n]^{S_{\mu(n)}}$
by averaging over $S_{\mu(n)}$:
\[
   \phi_n^\mu(f)
   \;=\;
   \frac{1}{|S_{\mu(n)}|}\sum_{\sigma\in S_{\mu(n)}}\sigma\cdot f.
\]
For increasing $n$, the images $I_n^{S_{\mu(n)}}$ form a directed system under the maps
$\phi_{n+1}^\mu\circ\iota_s$.  This is well-defined: since $\iota_s(f)\in I_{n+1}$ and
$I_{n+1}$ is $S_{\mu(n+1)}$-invariant, $\phi_{n+1}^\mu(\iota_s(f))\in I_{n+1}^{S_{\mu(n+1)}}$.

The \emph{block-symmetric limit} is the direct limit
\[
   I_\infty^\mu
   \;=\;
   \varinjlim\, I_n^{S_{\mu(n)}}
   \;\subset\;
   \Lambda\otimes_K R_\mu,
\]
where $\Lambda = \varinjlim K[X_n]^{\Sym(n)}$ as in Definition~\ref{def_sym_lim}.
For each $n$, we have a limit map, the \emph{block-symmetric total symmetrization}
\[
\phi_{\infty}^\mu: \,K[X_n] \rar \Lambda\otimes_K R_\mu.
\]
\end{definition}

\begin{lemma}\label{lemma_block_sym_dim}
For every $d\in\NN$, the block-symmetric dimension
\[
   \SD^\mu(K[X]/I,\,d)
   \;=\;
   \lim_{n\rar\infty}\dim_K\bigl[K[X_n]^{S_{\mu(n)}}/I_n^{S_{\mu(n)}}\bigr]_d
\]
exists and is finite.
\end{lemma}

\begin{proof}
The argument parallels that of Lemma~\ref{lemma_fin_sym_dim} and is a direct adaptation. Note that  the first tensor factor
$K[X_{n-n_0+\mu_1}]^{\Sym(n-n_0+\mu_1)}$ is generated by elementary symmetric polynomials
of degrees $1,\dots,n-n_0+\mu_1$. Therefore, the space $[K[X_n]^{S_{\mu(n)}}]_d$ is constant for
$n-n_0+\mu_1\ge d$, i.e., for $n\ge d+n_0-\mu_1$.  The inclusions
$\phi_{n+1}^\mu\circ \iota_{s}{:}\ I_n^{S_{\mu(n)}}\irar I_{n+1}^{S_{\mu(n+1)}}$ show that
$\dim_K[I_n^{S_{\mu(n)}}]_d$ is non-decreasing and bounded, hence eventually constant.
\end{proof}

The block-symmetric Hilbert series
\[
   H^{S_\mu}(K[X]/I;\,t)
   \;=\;
   \sum_{d\ge 0}\SD^\mu(K[X]/I,\,d)\,t^d
   \;=\;
   H\bigl(\Lambda\otimes R_\mu/I_\infty^\mu;\,t\bigr)
\]
is therefore a well-defined formal power series with non-negative integer coefficients.

\subsection{Rationality of the block-symmetric Hilbert series}

\begin{theorem}\label{thm_block_sym_rational}
Let $I\subset K[X]$ be a homogeneous symmetric ideal satisfying condition $(\dagger)$, and let
$\mu\vdash n_0$.  Then the block-symmetric Hilbert series $H^{S_\mu}(K[X]/I;\,t)$ is a rational
function.
\end{theorem}

\begin{proof}
The strategy follows the proof of Theorem~\ref{thm_main}, with the polynomial ring $R_\mu$
carried along as an additional finitely generated factor.

\medskip\noindent\textbf{Step~1: Joint-disjoint closedness of the initial space.}
We equip $\Lambda\otimes_K R_\mu$ with the monomial order that first compares the $\Lambda$-factor
using the reverse-length lexicographic order~$\ge$ from Section~\ref{subsec_revlex_order}, and breaks ties by degree
lexicographic order on $R_\mu$. For this we need to define a monomial basis on $R_\mu$.

Let $\ell$ be the length of $\mu$. By the theorem of symmetric polynomials, we can write the tensor factor of $R_\mu$ corresponding to $\Sym(\mu_k)$ explicitly as $K[e_1,\dots,e_{{\mu_k}}]$ where $e_i$ is the $i$-th elementary symmetric polynomial in $\mu_k$-variables. Thus we can write $R_\mu$ formally as a polynomial ring
$R_\mu \cong K\big[e_{i_k}^k\mid 2 \leq k \leq \ell,\ 1\leq i_k \leq \mu_k\big]$, where the variable $e_i^k$ has degree $i$. We obtain a degree lexicographical order on $R_\mu$ by some order of the variables $e_{i_k}^k$.

Given $f\in\Lambda\otimes R_\mu$, its \emph{initial monomial}
$\init(f)$ is the maximal monomial $\lambda\otimes\mathbf{r}$ appearing.

We claim that the initial space $\init(I_\infty^\mu)\subset\Lambda\otimes R_\mu$ is
\emph{joint-disjoint closed in the $\Lambda$-factor}: if $\lambda\otimes\mathbf{r}\in
\init(I_\infty^\mu)$ and $\nu$ is any partition, then
$(\nu\vee\lambda)\otimes\mathbf{r}\in\init(I_\infty^\mu)$ and
$(\nu\wedge\lambda)\otimes\mathbf{r}\in\init(I_\infty^\mu)$.

For the disjoint product, we adapt the argument of
Proposition~\ref{prop_initial_ideal_closed_cond}.  Given $f\in I_\infty^\mu$ with
$\init(f)=\lambda\otimes\mathbf{r}$ and a partition $\nu=(\nu_1,\dots,\nu_k)$, lift $f$
to $f_{(n)}\in I_n^{S_{\mu(n)}}$ for $n$ large.  Apply $\iota_s$ a total of $k$ times to
obtain $f_{(n)}(x_{k+1},\dots,x_{n+k})\in I_{n+k}$ (as a composition of $k$ shifted
inclusions, each realised by a permutation in $\Sym(\infty)$).  For $n$ large enough,
the variables $x_1,\dots,x_k$ all lie in Block~$1$ at stage~$n+k$
(since $|\text{Block}_1^{(n+k)}| = n+k-n_0+\mu_1\ge k$) and are disjoint from the support
of $f_{(n)}(x_{k+1},\dots,x_{n+k})$.  Set
\[
   g_{(n+k)}
   \;=\;
   \phi_{n+k}^\mu\bigl(f_{(n)}(x_{k+1},\dots,x_{n+k})\cdot
   x_1^{\nu_1}\cdots x_k^{\nu_k}\bigr)
   \;\in\; I_{n+k}^{S_{\mu(n+k)}}.
\]
Since the monomial $x_1^{\nu_1}\cdots x_k^{\nu_k}$ is supported in Block~$1$ and
disjoint from the support of the shifted $f_{(n)}$, the same argument as in
Proposition~\ref{prop_initial_ideal_closed_cond} gives $\init(g)=(\nu\vee\lambda)\otimes\mathbf{r}$.

For the joint product, we only need to consider the case $\len(\nu) \leq \len(\lambda)$. This follows from the same argument as in the proof of Proposition \ref{prop_initial_ideal_closed_cond}.  The ideal $I_n^{S_{\mu(n)}}$ is a module over the invariant
ring $K[X_{n-n_0+\mu_1}]^{\Sym(n-n_0+\mu_1)}$ of the first block.  Multiplication of $f_{(n)}$
by a monomial symmetric polynomial $m_\nu\in K[X_{n-n_0+\mu_1}]^{\Sym(n-n_0+\mu_1)}$ preserves
$I_n^{S_{\mu(n)}}$.  The product $m_\lambda m_\nu$ in the finitely generated ring has initial MSF proportional
to $m_{\lambda\wedge\nu}$, since $\len(\nu) \leq \len(\lambda)$ and the $R_\mu$-component of $f_{(n)}$ is unaffected. Passing to the
limit gives $(\nu\wedge\lambda)\otimes\mathbf{r}\in\init(I_\infty^\mu)$.

\medskip\noindent\textbf{Step~2: Rank bound from condition $(\dagger)$.}
Let $f\in I$ satisfy $(\dagger)$: $f=\sum_{i=1}^N c_i m_i$ with
$\Supp(m_i)\not\subset\Supp(m_1)$ for $i\ge 2$.  By symmetry of $I$, we may assume
$\Supp(m_1)=\{1,\dots,a\}\subset\{1,\dots,n-n_0+\mu_1\}$ for $n$ sufficiently large, so that
$m_1$ is supported entirely in the growing first block.  Write $m_1=x_1^{\alpha_1}\cdots
x_a^{\alpha_a}$ and set $b=1+\max\{\alpha_i\}$.  Then
\[
   f' \;=\;\phi_\infty^\mu\bigl(f\cdot x_1^{b-\alpha_1}\cdots x_a^{b-\alpha_a}\bigr)
   \;\in\; I_\infty^\mu
\]
has initial monomial $(b^a)\otimes 1$, since all other terms have strictly greater length in the
$\Lambda$-factor.  Extending $(b^a)$ by joint and disjoint products in $\Lambda$ shows that
$\lambda\otimes\mathbf{r}\in\init(I_\infty^\mu)$ for all $\lambda$ with
$\rk(\lambda)\ge M:=\max\{a,b\}$ and all $\mathbf{r}\in R_\mu$.

\medskip\noindent\textbf{Step~3: Rank decomposition.}
Exactly as in the proof of Theorem~\ref{thm_main}, we decompose
\[
   \Lambda\otimes R_\mu/\init(I_\infty^\mu)
   \;=\;
   \bigoplus_{m=0}^{M-1} R_m^\mu,
\]
where $R_m^\mu=\langle\lambda\otimes\mathbf{r}\notin\init(I_\infty^\mu)\mid\rk(\lambda)=m
\rangle_K$.  Each $R_m^\mu$ is a finitely generated graded module over
$A_m\otimes R_\mu\cong K[y_1,z_1,\dots,y_m,z_m]\otimes R_\mu$, which is a polynomial ring in
$2m+(n_0-\mu_1)$ variables.  By the Hilbert--Serre theorem, $H(R_m^\mu;\,t)$ is a rational
function, and
\[
   H^{S_\mu}(K[X]/I;\,t)
   \;=\;\sum_{m=0}^{M-1}H(R_m^\mu;\,t)
\]
is a finite sum of rational functions, hence rational.
\end{proof}

\subsection{Kostka inversion and isotypic rationality}

We now pass from block-symmetric invariants to individual multiplicities.
For an introduction to Kostka numbers, see \cite[Section 2.11]{Sag01}.

\begin{lemma}[Stable Kostka matrix]\label{lemma_stable_kostka}
Let $\mu,\nu\vdash n_0$.  Then for all $n$ sufficiently large (specifically, $n\ge 2n_0$), the
Kostka number $K_{\nu(n),\mu(n)}$ is independent of $n$.  We denote this stable value by
$\overline{K}_{\nu,\mu}$.
\end{lemma}

\begin{proof}
Recall that $K_{\nu(n),\mu(n)}$ counts the number of semi-standard Young tableaux (SSYT) of
shape $\nu(n)$ and content $\mu(n)$.  The content $\mu(n)$ has $n-n_0+\mu_1$ entries equal to~$1$,
followed by $\mu_2$ entries equal to~$2$, etc.

Since columns of a SSYT are strictly increasing, all entries equal to~$1$ must appear in the first row. In particular, the inequality $\nu_1\ge\mu_1$ is necessary for $K_{\nu(n),\mu(n)}\ne 0$.
For $n \geq 2n_0$, the first $n-n_0 + \mu_1 > n_0 \geq \nu_2$ entries of the first row of the SSYT of shape $\nu(n)$ are forced to be $1$. Removing the first cell of the first row thus yields a SSYT of shape $\nu(n-1)$ and content $\mu(n-1)$, since the columns are still strictly increasing. This is the case as the condition of strictly increasing columns is only nontrivial for the first $\nu_2$ cells of the first row.

This removal of the first cell defines a bijection of SSYTs of shape $\nu(n)$ and content $\mu(n)$ to SSYTs of shape $\nu(n-1)$ and content $\mu(n-1)$.
Therefore the Kostka numbers are equal and we have $K_{\nu(n),\mu(n)} = K_{\nu(2n_0),\mu(2n_0)}$ for $n \geq 2n_0$. 
\end{proof}

The matrix $\overline{K}=(\overline{K}_{\nu,\mu})_{\nu,\mu\vdash n_0}$ is unitriangular with
respect to the dominance order on partitions of $n_0$, and hence invertible over $\ZZ$.

\begin{proposition}[Frobenius reciprocity]\label{prop_frobenius}
Let $M$ be any finite-dimensional $\Sym(n)$-representation.  Then
\[
   \dim_K M^{S_{\mu(n)}}
   \;=\;
   \sum_{\nu\vdash n} K_{\nu,\mu(n)}\, c_\nu(M),
\]
where $c_\nu(M)$ is the multiplicity of $V^\nu$ in $M$.
\end{proposition}

\begin{proof}

This is a standard consequence of Frobenius reciprocity. Let $\chi$ be the character of the $\Sym(n)$-representation $M$, $\chi_\nu$ the character of the irreducible representation $V^\nu$, and $\bm 1$ the character of the trivial representation of $S_{\mu(n)}$. We use the up and down arrow notation for characters of restricted and induced representations following \cite{Sag01}. Then
\begin{align*}
    \dim_K M^{S_{\mu(n)}} &= \langle \bm{1},\, \chi{\downarrow}_{S_{\mu(n)}} \rangle_{S_{\mu(n)}} \\
    &= \langle \bm{1}{\uparrow}^{\Sym(n)},\, \chi \rangle_{\Sym(n)} \\
    &= \Big\langle\sum_{\nu \vdash n} K_{\nu,\mu(n)}\chi_\nu ,\, \chi \Big\rangle_{\Sym(n)} \\
    &= \sum_{\nu\vdash n} K_{\nu,\mu(n)} \, \langle\chi_\nu,\,\chi\rangle_{\Sym(n)},
\end{align*}
and $\langle \chi_\nu,\chi\rangle = c_\nu(M)$. The second equality is Frobenius reciprocity and the third equality is due to the well-known decomposition of the permutation module, see \cite[Theorem 2.11.2]{Sag01}. 
\end{proof}

\begin{theorem}\label{thm_isotypic_rational}
Let $I\subset K[X]$ be a homogeneous symmetric ideal satisfying condition~$(\dagger)$.  Then
for every partition $\mu\vdash n_0$ the isotypic Hilbert series $H^\mu(K[X]/I;\,t)$ is a rational
function.
\end{theorem}

\begin{proof}
Fix some $\mu \vdash n_0$. We may assume that $\mu_1 \geq n_0 - \mu_1$, simply by replacing $\mu$ with $\mu(2n_0)$ and noting that this does not change the corresponding block-symmetric Hilbert series.

By Proposition~\ref{prop_frobenius}, we have for $n\geq n_0$ a decomposition of the $S_{\mu(n)}$-invariant submodule
\[
\dim_K\big[ K[X_n]/I_n\big]_d^{S_{\mu(n)}} = \sum_{\nu\vdash n} K_{\nu,\mu(n)}c_{\nu,d}(n)
\]
where, as usual, $c_{\nu,d}(n)$ denotes the multiplicity of the irreducible representation $V^\nu$ in $\big[K[X_n]/I_n\big]_d$. If $\nu$ does not dominate the partition $\mu(n)$, then $K_{\nu,\mu(n)} = 0$. We can exclude all such cases from our summation. Otherwise, we have $\nu \trianglerighteq \mu(n)$ and in particular $\nu_1 \geq \mu_1 +n-n_0$. Using our assumption $\mu_1 \geq n_0 - \mu_1$ in the second step and $n \geq \nu_1 +  \nu_2$ in the third step, we get
\begin{align*}
    \nu_1 &\geq \mu_1 +n-n_0 \\
    &\geq  n_0 - \mu_1 + n - n_0 \\
    &\geq -\mu_1 + \nu_1 + \nu_2 \\
    &\geq -\mu_1 + \mu_1 +n - n_0 + \nu_2 \\
    &= \nu_2 + n - n_0.
\end{align*}
Hence we can define $\nu' := (\nu_1 - n + n_0,\, \nu_2,\,\nu_3,\,\dots)$, which is a partition of $n_0$. We have $\nu'(n) = \nu$ by construction.

Following this argument, we see that for any $n$ large enough,
\[
\dim_K \big[ K[X_n]/I_n\big]_d^{S_{\mu(n)}} = \sum_{\nu\vdash n_0} K_{\nu(n),\mu(n)}c_{\nu(n),d}(n)
\]
which is a sum with a constant number of terms as $n$ grows. This lets us compute the limit for $n\rar \infty$, in terms of block-symmetric dimensions and stable Kostka numbers,
\begin{equation}\label{eq_kostka_system}
\SD^{\mu}(K[X]/I,\,d)
   \;=\;
   \sum_{\nu\vdash n_0}\overline{K}_{\nu,\mu}\;\lim_{n\rar\infty}c_{\nu(n),d}(n)
\end{equation}

The system \eqref{eq_kostka_system} for all $\mu \vdash n_0$ has coefficient matrix $(\overline{K}_{\nu,\mu})$, which is unitriangular with respect to
dominance order and hence invertible over $\ZZ$.
Writing the inverse as
$(\overline{K}^{-1})_{\mu,\nu}\in\ZZ$, we obtain
\[
   \lim_{n\rar\infty}c_{\nu(n),d}(n)
   \;=\;
   \sum_{\mu\vdash n_0}(\overline{K}^{-1})_{\mu,\nu}\;\SD^{\mu}(K[X]/I,\,d).
\]
As a generating function in~$t$:
\[
   H^{\nu}(K[X]/I;\,t)
   \;=\;
   \sum_{\mu\vdash n_0}(\overline{K}^{-1})_{\mu,\nu}\;H^{S_\mu}(K[X]/I;\,t).
\]
By Theorem~\ref{thm_block_sym_rational}, each $H^{S_\mu}(K[X]/I;\,t)$ is rational. A finite
$\ZZ$-linear combination of rational functions is rational.
\end{proof}

\begin{remark}\label{rem_isotypic}
\begin{enumerate}
\item The case $\nu=(1^{n_0})$ yields the isotypic Hilbert series for the sign representation:
   $\nu(n)=(n-n_0+1,1^{n_0-1})$, a hook partition.  This is the second ``extremal'' case
   mentioned in the introduction.
\item The Kostka inversion expresses individual multiplicities as integer linear combinations
   of block-symmetric dimensions. Since the inverse Kostka matrix has entries in $\ZZ$, the coefficients of $H^\nu$ are integers, as expected from representation
   theory.
\item An analogous result holds under the geometric condition of Corollary~\ref{geometric_cor}:
   if $V_{\overline{K}}(I)\subset\Delta_p\cup H_q$, then all isotypic Hilbert series are rational.
\end{enumerate}
\end{remark}

\begin{theorem}\label{thm_iso_polynomial_growth}
    Let $I\subset K[X]$ be any nonzero homogeneous symmetric ideal. Fix any partition $\mu \vdash n_0$. Then the corresponding stable multiplicities have at most polynomial growth,
    \[
    \lim_{n\rar \infty} c_{\mu(n), d}(n) = \mathcal{O}(d^k)
    \]
    for some $k\in \NN$ depending on $I$ and $\mu$.
\end{theorem}

\begin{proof}
    Let $I\subset K[X]$ be any nonzero homogeneous symmetric ideal and fix some partition $\mu\vdash n_0$. We write $\mu = (\mu_1,\dots,\mu_\ell)$. If $\ell = 1$ we are in the case of the symmetric Hilbert series and Corollary \ref{cor_sym_polynomial_growth} gives the result.
    
    Assume that $\ell \geq 2$. Following the proof of Theorem \ref{thm_isotypic_rational}, we see that the stable multiplicity value $\lim_{n\rar\infty} c_{\mu(n),d}(n)$ is a linear combination of block-symmetric dimensions $\SD^\nu(K[X]/I,d)$. The coefficients and partitions $\nu$ appearing in this linear combination are independent of $d$. It therefore suffices to show polynomial growth of block-symmetric dimensions of $K[X]/I$.

    As in the proof of Corollary \ref{cor_sym_polynomial_growth}, let $r$ be the minimal degree of nonzero homogeneous polynomials appearing in $I$. Then the symmetric dimension grows at most polynomially,
    \[
    \SD(K[X]/I,d) = \mathcal{O}(d^{2r-1}).
    \]

    We denote $s = n_0 - \mu_1$ and $\nu = (\mu_2,\dots,\mu_\ell)\vdash s$ for the \emph{tail partition} of $\mu$. By the assumption $\ell \geq 2$, it is not the zero partition and $s\geq 1$.
    Consider the map of graded vector spaces,
    \begin{gather*}
    F: K[X_{n-s}]^\Sym/I_{n-s}^\Sym\otimes_K K[x_{n-s+1},\dots,x_n]^{S_{\nu}} \longrightarrow K[X_n]^{S_{\mu(n)}}/I_{n}^{S_{\mu(n)}} \\
        \overline{f_1}\otimes f_2 \mapsto \overline{f_1f_2} 
    \end{gather*}
    The map is clearly linear and graded. We check that this map is well defined and surjective. If $f_1\in I_{n-s}^\Sym$ and $f_2 \in K[X_{s}]^{S_\nu}$, then $f_1f_2 \in I_n^{S_{\mu(n)}}$ since $I_{n-s}^\Sym \subset I_n^{S_{\mu(n)}}$ as a vector space. This shows that the map is well defined on equivalence classes.

    For surjectivity, note that $K[X_{n-s}]^{\Sym} \otimes_K K[X_s]^{S_{\nu}} \cong K[X_n]^{S_{\mu(n)}}$, hence we only need to show that all equivalence classes of elementary tensors of the form $f_1\otimes f_2$ lie in the image of $F$. This is clear from the definition.

    Since $F$ is a graded linear map, we obtain an inequality of dimensions
    \begin{align*}
        \dim_K \big[ K[X_n]^{S_{\mu(n)}}/I_n^{S_{\mu(n)}} \big]_d &\leq \dim_K\big[ K[X_{n-s}]^\Sym/I_{n-s}^{\Sym}\otimes_K K[x_{n-s+1},\dots,x_n]^{S_{\nu}} \big]_d \\
        &= \sum_{i= 0}^d \dim_K\big[K[X_{n-s}]^\Sym/I_{n-s}^\Sym\big]_i\cdot\dim_K\big[ K[X_s]^{S_\nu} \big]_{d-i}. 
    \end{align*}

    For $n\rar\infty$ we compare the generating function of the sequences and see that the growth of the block symmetric dimension $\SD^\mu(K[X]/I,d)$ is at most polynomial since $\SD(K[X]/I,d) = \mathcal{O}(d^{2r-1})$ by Corollary \ref{cor_sym_polynomial_growth} and $\dim_K\big[ K[x_{n-s+1},\dots,x_n]^{S_\nu} \big]_{d} = \dim_K\big[ K[X_s]^{S_\nu}\big]_{d} = \mathcal{O}(d^{s-1})$. This concludes the proof.
\end{proof}

\subsection{Final remarks}\label{subsec_final_rem}
We close this article with a discussion of rationality of the isotypic Hilbert series outside of condition $(\dagger)$ and possible directions of future work.

The zero ideal $I = (0) \subset K[X]$ is certainly a symmetric ideal and it is not difficult to compute its symmetric Hilbert series
\[
H^{\Sym}(K[X];\,t) = H(\Lambda;\,t) = \prod_{d=1}^\infty \frac{1}{1-t^d}
\]
which is not a rational function. Using Kostka inversion and the block-symmetric Hilbert series, it follows that the isotypic Hilbert series $H^\mu(K[X];\,t)$ is not a rational function for any partition $\mu$ either.

However, counterexamples for rationality of the isotypic Hilbert series begin and end here. Moreover, this behaviour should be expected: $K[X]$ is the coordinate ring of \emph{infinite} affine space $\bA^\infty_K$ while the vanishing set of any nonzero symmetric ideal $I \subset K[X]$ corresponds to a \emph{finite} dimensional subvariety of $\bA_K^\infty$. This is a fundamental difference between $I = (0)$ and all other symmetric ideals in $K[X]$.

The main ingredient in the proof of Theorem \ref{thm_main} was the decomposition of $\Lambda/\init(I_\infty)$ into finitely many spaces of fixed rank, each equipped with a different module structure over a finite polynomial ring.
This is mirrored in the proof of Theorem \ref{thm_hook_specht_main} where we decompose $\Lambda/\init(I_{\infty})$ into parts of fixed $<_1$-chain divisibility. 
Here, there is no clear module structure for each $P_k$, but we were able to show rationality of each part by a complicated counting argument. That such a counting argument yields a rational generating function should be taken as evidence that there is more underlying structure to be uncovered.
Moreover, it seems plausible that once the structure of the hook Specht ideal quotients is better understood, it would allow for the extension of rationality to all symmetric ideals containing a hook Specht ideal. 

Taking the above discussion as well as the lack of counterexamples into consideration, we propose the following.
\begin{conjecture}\label{con_rational}
    Let $I\subset K[X]$ be any homogeneous symmetric ideal. Then the isotypic Hilbert series $H^\mu(K[X]/I;t)$ is a rational function for any partition 
    $\mu$ if $I \neq (0)$.
\end{conjecture}

For future work, a formulation of the theory of isotypic Hilbert series for finitely generated graded modules of graded $FI$-algebras would be of interest.
In particular, the infinite polynomial ring with a diagonal permutation action on more than one infinite set of variables is an interesting generalization.

The polynomial ring in $k$ infinite sets of variables $K[X^1,\,\dots,\,X^k]$ with $X^j = \{x_{i,j}\mid i\geq 1\}$ has a natural action of $\Sym(\infty)$ on the first index given by $\sigma\cdot x_{i,j} = x_{\sigma(i),j}$.
The definitions for the symmetric and isotypic Hilbert series generalize without much difficulty to quotients of such rings. This is also the more general setup in works such as \cite{KLS17,KR25,NR17} and ideals closed under this $\Sym(\infty)$-action still satisfy the ascending chain condition.

We can, however, immediately note that we have no hope of Conjecture \ref{con_rational} holding for such a group action.
Nonetheless, it seems reasonable to believe that the isotypic Hilbert series will allow for a finite description in terms of rational functions and a finite set of ``basic" formal power series as building blocks.
In particular, general rationality of the isotypic Hilbert series of $K[X^1,\dots,X^k]/I$ under the assumption that the vanishing set $V(I)$ is finite dimensional seems possible.

We illustrate this with three examples for the infinite polynomial ring in two sets of variables.

\begin{example}
	Denote the ``basic" formal power series
    \[
    P(t) = \prod_{d = 1}^\infty \frac{1}{1-t^d}.
    \]

    We consider homogeneous and $\Sym(\infty)$-closed ideals in $K[X,Y] = K[x_{i},y_i \mid i \geq 1]$
    \[
    I^{(1)} = (x_iy_j \mid i,j \geq 1),\quad I^{(2)} = (x_iy_i \mid i \geq 1),\quad I^{(3)} = (x_ix_j, \,y_iy_j \mid i,j\geq 1, i\neq j).
    \]

    As in the case of only one set of variables, we denote the truncation at the first $n$ variables for $X$ and $Y$  by $K[X_n,Y_n]$. Then $I^{(\ell)}_n = I^{(\ell)} \cap K[X_n,Y_n]$ for $\ell = 1,2,3$.
    
    The ring $K[X_n,Y_n]/I^{(1)}_n$ contains exactly those monomials solely in either $x_i$ or $y_i$. Thus
    \[
    \big[K[X_n,Y_n]/I^{(1)}_n\big]_d \cong \big[K[X_n]\big]_d\oplus \big[K[Y_n]\big]_d
    \]
    for $d\geq 1$ and we can conclude that $H^\Sym(K[X,Y]/I^{(1)};\,t) = 2P(t) -1$, since $P(t)$ is the symmetric Hilbert series of $K[X]$.

    The analysis for the second case is subtly different. The quotient $K[X_n,Y_n]/I^{(2)}_n$ contains exactly those monomials not divisible by both $x_i$ and $y_i$ for each $i$. Averaging such a monomial over $\Sym(n)$ always gives a product of monomial symmetric polynomials 
    \[
    m_\lambda(x_1,\dots,x_n)m_\mu(y_1,\dots,y_n)
    \]
    modulo $I^{(2)}_n$. The limit is therefore isomorphic to the tensor product $\Lambda \otimes_K \Lambda$. This finally gives us the symmetric Hilbert series $H^\Sym(K[X,Y]/I^{(2)};\,t) = P(t)^2$.

    For the third case, we note that, excluding constants, all monomials in $K[X_n,Y_n]/I^{(3)}_n$ are either of the form $x_i^a, y_i^b$ for $a,b \geq 1$, which we call first form, or of the form $x_i^ay_i^b,x_i^ay_j^b$ for $i \neq j$ and $a,b \geq 1$, which we call second form. This holds for any $n \geq 2$, and averaging over $\Sym(n)$ lets us find a basis for $\big(K[X_n,Y_n]/I^{(3)}_n\big)^\Sym$. This yields the Hilbert series
    \[
    H\Big({\big(K[X_n,Y_n]/I^{(3)}_n}\big)^\Sym;\,t\Big) = 1 +\underbrace{ \frac{2t}{1-t} }_\textnormal{first form} + \underbrace{ \frac{2t^2}{(1-t)^2} }_\textnormal{second form} = \frac{1+t^2}{(1-t)^2}.
    \]
    This Hilbert series is stable for all $n \geq 2$ and thus equal to the symmetric Hilbert series $H^\Sym(K[X,Y]/I^{(3)};\,t)$. It is in particular a rational function.
\end{example}
The three ideals suggest that in the multisymmetric setting rationality is governed by
one geometric quantity: the dimension of the set of values a sequence attains. For a
point $p\in V(I)(\overline K)$, viewed as a sequence $(p_i)_{i\ge1}$ in
$\bA^k(\overline K)$ with $\Sym(\infty)$ permuting the index $i$, let
$V_p\subseteq\bA^k_{\overline K}$ be the Zariski closure of its set of values
$\{p_i\mid i\ge1\}$, the \emph{value variety} of $p$ (cf.\ \cite[Def.~3.2]{KR25}, whose
ambient dimension $n$ is our number of variable sets $k$), and set
$\dim V(I):=\sup_p\dim V_p\le k$. Since a zero-dimensional closed subset of $\bA^k$ is
finite, $\dim V(I)=0$ holds precisely when every point of $V(I)$ attains only finitely
many values.

The transcendental contribution of a $d$-dimensional value component is a universal
series. For $d\ge0$ write
\[
   P_d(t) := \prod_{m\ge 1}\bigl(1-t^m\bigr)^{-\binom{m+d-1}{d-1}}.
\]
This is the symmetric Hilbert series $H^\Sym(K[X^1,\dots,X^d];t)$ of the free quotient
on $d$ sets of variables: a basis of the invariants is given by the finite multisets of
nonzero exponent vectors $\alpha\in\NN^d$, of which $\binom{m+d-1}{d-1}$ have degree
$m=|\alpha|$. In particular $P_0=1$ and $P_1=P$, and each $P_d$ with $d\ge1$ is
non-rational, having a pole at every root of unity.

\begin{conjecture}\label{con_rational_strong}
Let $I\subset K[X^1,\dots,X^k]$ be a homogeneous $\Sym(\infty)$-invariant ideal. Then
for every partition $\mu$,
\[
   H^{\mu}\bigl(K[X^1,\dots,X^k]/I;\,t\bigr)
   = \sum_{j} R_j(t)\,\prod_{i} P_{d_{ji}}(t),
\]
with $R_j\in\QQ(t)$ and each $d_{ji}\ge 1$ the dimension of a positive-dimensional
component of a value variety of $V(I)$. In particular $H^\mu$ is rational for every
$\mu$ if and only if $\dim V(I)=0$.
\end{conjecture}

This fits every computed case: $I^{(1)},I^{(2)}$ give $2P_1-1$ and $P_1^2$, the sum and
the product reflecting whether the two branches of $\{XY=0\}$ are populated by one
global site or independently, while $I^{(3)}$ gives a rational series. It also recovers
Conjecture~\ref{con_rational} at $k=1$: a point with $\dim V_p\ge 1$ has $V_p=\bA^1$,
whose orbit closure is all of $\bA^1_\infty$ \cite[Thm.~3.5]{KR25}, so $\dim V(I)\ge1$
forces $I=(0)$; equivalently, by the classification of invariant primes
\cite[Ex.~3.13]{KR25}, the zero ideal is the only one of positive dimension. This is
what singles out $(0)$ at the start of the section. The sufficient direction seems
within reach, since for $\dim V(I)=0$ the methods of Theorem~\ref{thm_main} and
Section~\ref{sec_isotypic} should apply with little change; necessity would follow from
the $P_d$, $d\ge1$, being non-cancellable in the $\ZZ$-combinations from Kostka
inversion, which we leave open.
\section{Conclusion and open questions}\label{sec_open_quest}
We have shown that the symmetric and isotypic Hilbert series of a homogeneous symmetric
ideal are rational under condition~$(\dagger)$, and the same for the symmetric Hilbert series
for the hook Specht ideals outside of it; along the way the dimension $\dim V(I)$ of the value variety emerged as the quantity
that should govern rationality in general. Several natural questions remain.

The most immediate concerns the mechanism behind rationality itself. Under
condition~$(\dagger)$ we obtained the symmetric Hilbert series as an instance of the
Hilbert--Serre theorem, by exhibiting $\Lambda/\init(I_\infty)$ as a finite direct sum
of finitely generated modules over polynomial rings. For the hook Specht ideals no such
module structure was available, and rationality came instead from an explicit counting
argument over the strata~$P_k$. That two such different routes lead to the same
conclusion suggests an underlying structure we have not yet identified. It would be
highly interesting to place the hook case---and, more ambitiously, every ideal with
$\dim V(I)=0$---on a common module-theoretic footing, realising $\Lambda/I_\infty$ as a
finitely generated graded module over a polynomial ring whose dimension can be read off
the geometry of~$V(I)$. Such a description would account for rationality uniformly
rather than case by case, and would render the polynomial growth rate of
$\SD(K[X]/I,d)$ a computable invariant of that geometry. At present we have such a
structure only under~$(\dagger)$, and the hook computation stands as evidence that it
should exist more widely.

The hook ideals also point to a concrete next computation. They are the Specht ideals
of the shapes $(\infty,1^r)$, and it is natural to ask whether the two-row shapes
$(\infty,2^a, 1^b)$, and eventually arbitrary shapes, admit a comparable treatment. Our
argument rested on identifying the symmetric limit with a determinantal ideal of an
infinite Hankel matrix, for which a Gröbner basis is known; whether the higher Specht
ideals carry an analogous structured-matrix description seems worth pursuing, and would
considerably enlarge the class of ideals whose symmetric Hilbert series can be written
down explicitly.

Finally, the theory invites a real counterpart. For a symmetric basic closed
semi-algebraic set the cohomology stabilises as an $\Sym(n)$-module \cite{BR20}, and one
may ask whether the resulting stable Betti-number multiplicities assemble into an
isotypic Poincaré series, rational under a dimension hypothesis of the kind isolated
here. In a different direction, the whole construction rests on the ring of symmetric
functions and the combinatorics of partitions; it would be interesting to see how much
survives when $\Sym(\infty)$ is replaced by the hyperoctahedral tower $B_\infty$, where
partitions give way to pairs of partitions. We hope to return to some of these questions
elsewhere.
\printbibliography

\end{document}